\documentclass[a4paper,11pt,fullpage]{article}
\usepackage[english]{babel}

\usepackage[left=1in,right=1in,top=1in,bottom=1in]{geometry}

\usepackage{amsmath}
\usepackage{amsthm}
\usepackage{amsfonts}
\usepackage{amssymb}
\usepackage{amsxtra}
\usepackage{latexsym}
\usepackage{ifthen}
\usepackage{cite}
\usepackage{esint}
\usepackage[cp1250]{inputenc}

\usepackage{epic}
\usepackage{eepic}
\usepackage{xcolor}

\DeclareMathOperator*{\esssup}{ess\,sup}

\DeclareMathOperator*{\essinf}{ess\,inf}

\DeclareMathOperator*{\osc}{osc}

\numberwithin{equation}{section}
\newtheorem{theorem}{Theorem}[section]
\newtheorem{lemma}{Lemma}[section]
\newtheorem{remark}{Remark}[section]
\newtheorem{definition}{Definition}[section]

\def\XXint#1#2#3{{\setbox0=\hbox{$#1{#2#3}{\int}$}
     \vcenter{\hbox{$#2#3$}}\kern-.5\wd0}}

\begin{document}

\title{Harnack inequality for solutions of the  $p(x)$-Laplace equation under the precise non-logarithmic Zhikov's conditions
}

\author{ Igor I. Skrypnik, \ Yevgeniia A. Yevgenieva
 }

  \maketitle

  \begin{abstract}
We prove continuity and Harnack's inequality for bounded solutions to the equation
$$
{\rm div}\big(|\nabla u|^{p(x)-2}\,\nabla u \big)=0,
\quad p(x)= p + L\frac{\log\log\frac{1}{|x-x_{0}|}}{\log\frac{1}{|x-x_{0}|}},\quad L > 0,
$$
under the precise non-logarithmic condition on the function $p(x)$.



\end{abstract}

\pagestyle{myheadings} \thispagestyle{plain}
\markboth{Igor I. Skrypnik}
{ Harnack inequality for solutions of the  $p(x)$-Laplace equation . . . .}

\section{Introduction and main results}\label{Introduction}

Let $\Omega$ be a bounded domain in $\mathbb{R}^{n}$, $n\geqslant2$.
In this paper we are concerned with elliptic equations of the type
\begin{equation}\label{eq1.1}
{\rm div} \mathbf{A}(x, \nabla u)=0, \quad x\in\Omega.
\end{equation}
We suppose that the functions $\mathbf{A}:\Omega\times\mathbb{R}^{n}\rightarrow\mathbb{R}^{n}$
are such that $\mathbf{A}(\cdot,\xi)$ are Lebesgue measurable for all $\xi\in \mathbb{R}^{n}$,
and $\mathbf{A}(x,\cdot)$ are continuous for almost all $x\in\Omega$.
We assume also that the following structure conditions are satisfied
\begin{equation}\label{eq1.2}
\begin{aligned}
\mathbf{A}(x,\xi)\,\xi&\geqslant K_{1}\,|\xi|^{p(x)},
\\
|\mathbf{A}(x,\xi)|&\leqslant K_{2}\,|\xi|^{p(x)-1},
\end{aligned}
\end{equation}
where $K_{1}$, $K_{2}$ are positive constants, $p(x)=p+p(|x-x_{0}|),\quad p(|x-x_{0}|)=L\dfrac{\log\log\frac{1}{|x-x_{0}|}}{\log\frac{1}{|x-x_{0}|}}$ and $L > 0$.

The aim of this paper is to establish basic qualitative properties such as
continuity of bounded solutions and Harnack's inequality
for non-negative bounded solutions to equation \eqref{eq1.1}.

Before formulating the main results, we say few words
concerning the history of the problem. The study of regularity of minima
of functionals with non-standard growth has been initiated by Zhikov
\cite{ZhikIzv1983, ZhikIzv1986, ZhikJMathPh94, ZhikJMathPh9798, ZhikKozlOlein94},
Marcellini \cite{Marcellini1989, Marcellini1991}, and Lieberman \cite{Lieberman91},
and in the last thirty years, the qualitative theory
of second order elliptic and parabolic equations with so-called ''logarithmic'' condition, i.e. if
\begin{equation}\label{eq1.3}
\osc\limits_{B_{r}(x_{0})}\,p(x) \leqslant \frac{L}{\log\frac{1}{r}},\quad 0< r< 1,\quad 0< L< \infty,
\end{equation}
 has been actively developed
(see e.g.
\cite{Alhutov97, AlhutovMathSb05, AlhutovKrash04, AlkhSurnApplAn19, AlkhSurnJmathSci20, BarColMing, BarColMingStPt16, BarColMingCalc.Var.18, BenHarHasKarp20, BurchSkrPotAn, ColMing218, ColMing15, ColMingJFnctAn16, DienHarHastRuzVarEpn, Fan1995, FanZhao1999,
HarHastOrlicz, HarHasZAn19, HarHastLee18, HarHastToiv17, HarKuuLukMarPar, Krash2002, Ok, SkrVoitUMB19, SkrVoitNA20, Skr, Sur}
for references). Equations of this type and systems of such equations arise in various
problems of mathematical physics (see e.g. the monographs
\cite{AntDiazShm2002_monogr, HarHastOrlicz, Ruzicka2000, Weickert} and references therein).

The case when condition \eqref{eq1.3} is replaced by the condition
\begin{equation}\label{eq1.4}
\osc\limits_{B_{r}(x_{0})}\,p(x) \leqslant \frac{\mu(r)}{\log\frac{1}{r}},\quad \lim\limits_{r\rightarrow 0}\mu(r)=\infty,\quad
\lim\limits_{r\rightarrow 0}\frac{\mu(r)}{\log\frac{1}{r}}=0,
\end{equation}
 differs substantially from the
logarithmic case. It turns out that such non-logarithmic condition is a precise condition for the smoothness of finite functions in the corresponding Sobolev space $W^{1,p(x)}(\Omega)$. Thus this case is extremely interesting to study. But to our knowledge there are only few results in this direction.
Zhikov \cite{ZhikPOMI04}
obtained a generalization of the logarithmic condition which guaranteed the density
of smooth functions in Sobolev space $W^{1,p(x)}(\Omega)$.  Particularly,
this result holds if $1<p\leqslant p(x)$ and
\begin{equation}\label{eq1.5}
\osc\limits_{B_{r}(x_{0})}\,p(x) \leqslant L\,\frac{\log\log\frac{1}{r}}{\log\frac{1}{r}},\quad 0< L \leqslant \frac{p}{n}.
\end{equation}
Later Zhikov and Pastukhova \cite{ZhikPast2008MatSb}
 proved higher integrability
of the gradient of solutions to the $p(x)$-Laplace equation under the same condition.
Interior continuity, continuity up to the boundary and Harnack's inequality
to the $p(x)$-Laplace equation were proved in \cite{AlhutovKrash08}, \cite{AlkhSurnAlgAn19} and  \cite{SurnPrepr2018}
under condition \eqref{eq1.4} and
\begin{equation}\label{eq1.6}
\int\limits_{0}\,\exp\big(-\gamma\,\exp(\mu^{c}(r))\big)\frac{dr}{r} = +\infty,
\end{equation}
with some numbers $\gamma, c >1$. For example, the function $\mu(r)= L\,\log\log\log\dfrac{1}{r}$ stisfies conditions \eqref{eq1.4},
\eqref{eq1.6}, provided that $L$ is a sufficiently small positive number.

These results were generalized in \cite{SkrVoitNA20, ShSkrVoit20}
for a wide class of elliptic and parabolic equations with non-logarithmic Orlicz growth.
Particularly, it was proved in \cite{SkrVoitNA20} that under conditions \eqref{eq1.4}, \eqref{eq1.6} functions from the correspondent
De\,Giorgi's $\mathcal{B}_{1}(\Omega)$ classes are continuous and moreover, it was shown
that the solutions of the correspondent elliptic and parabolic equations with non-standard
growth belong to these classes.

The exponential condition of the type \eqref{eq1.6} was substantially refined in \cite{HadSkrVoi}. Particularly, the continuity of solutions
to double-phase and degenerate double-phase elliptic equations
\begin{equation*}
div\bigg(|\nabla u|^{p-2}\nabla u+a(x)|\nabla u|^{q-2}\nabla u\bigg)=0,\quad q>p,
\end{equation*}
and
\begin{equation*}
div\bigg(|\nabla u|^{p-2}\nabla u \big(1+\log(1+b(x)|\nabla u|)\big)\bigg)=0
\end{equation*}
was proved under the conditions
\begin{equation*}
\osc\limits_{B_{r}(x_{0})}\,a(x)\leqslant A\,\mu(r)^{q-p}\,r^{q-p},\quad \osc\limits_{B_{r}(x_{0})}\,b(x)\leqslant B\mu(r)\,r,\quad
\int\limits_{0}\frac{dr}{\mu(r)}=+\infty.
\end{equation*}
Note that the function $\mu(r)=\log\dfrac{1}{r}$ satisfies the above conditions.
In the present paper the continuity and the Harnack's type inequality have been proved under the conditions
similar to \eqref{eq1.5}.

Before formulating the main results, let us recall the definition of a bounded weak solution
to equation \eqref{eq1.1}. We introduce $W(\Omega)$ as a class of functions
$u\in W^{1,1}(\Omega)$, such that $\int\limits_{\Omega}|\nabla u|^{p(x)}\,dx<+\infty$, and $W_{0}(\Omega) =
W(\Omega) \cap W^{1,1}_{0}(\Omega).$

\begin{definition}
{\rm
We say that a function $u\in W(\Omega)\cap L^{\infty}(\Omega)$ is a bounded weak
sub(super)-solution
to equation \eqref{eq1.1} if 
\begin{equation}\label{eq1.7}
\int_{\Omega}\mathbf{A}(x, \nabla u)\,\nabla\varphi \,dx\leqslant(\geqslant)\,0,
\end{equation}
holds for all non-negative test functions $\varphi\in W_{0}(\Omega)$.
}
\end{definition}

The following Theorem is the first main result of this paper.
\begin{theorem}\label{th1.1}
Let $u$ be a bounded weak solution of equation \eqref{eq1.1} and let conditions \eqref{eq1.2}, \eqref{eq1.3} be fulfilled,
then $u$ is H\"{o}lder continuous at point $x_{0}$.
\end{theorem}
The next result is a weak Harnack type inequality for non-negative super-solutions.
\begin{theorem}\label{th1.2}
Let $u$ be a bounded non-negative weak super-solution to equation \eqref{eq1.1}, let conditions \eqref{eq1.2}, \eqref{eq1.3} be fulfilled.
Assume also that
\begin{equation}\label{eq1.8}
\big(\mathbf{A}(x,\xi)-\mathbf{A}(x,\eta)\big)(\xi-\eta) >0,\quad \xi, \eta \in \mathbb{R}^{n},
\quad \xi\neq\eta,
\end{equation}
then there exist numbers $\gamma, \bar{\gamma} >0$ depending only on $n, p, K_{1}, K_{2}$ and $M=\sup\limits_{\Omega} u$,
such that for any $\theta\in (0, p-1)$ there holds
\begin{equation}\label{eq1.9}
\bigg(|B_{\rho}(x_{0})|^{-1}\,\int\limits_{B_{\rho}(x_{0})}\,u^{\theta}\,dx\bigg)^{\frac{1}{\theta}} \leqslant
\bigg(\frac{\gamma}{p-1-\theta}\bigg)^{\frac{1}{\theta}}\, \big( \min\limits_{B_{\frac{\rho}{2}}(x_{0})} u +\,\rho\big),
\end{equation}
provided that $B_{16\rho}(x_{0})\subset \Omega$ and
\begin{equation}\label{eq1.10}
\frac{1}{\log\log\dfrac{1}{16\rho}} + \bar{\gamma}\,L\,\frac{\log\log\dfrac{1}{16\rho}}{\log\dfrac{1}{16\rho}} \leqslant 1.
\end{equation}
\end{theorem}
The following Theorem is  Harnack's inequality
\begin{theorem}\label{th1.3}
Let  $u$ be a bounded non-negative  weak sub-solution to equation \eqref{eq1.1}, let conditions \eqref{eq1.2}, \eqref{eq1.3} be fulfilled.
Then  there exist positive numbers $\gamma $, $\bar{\gamma}_{1}$ depending only on $n, p, K_{1},\\ K_{2}, M$
such that for any $\theta\in (0,p-1)$
\begin{equation}\label{eq1.11}
\max\limits_{B_{\frac{\rho}{2}(x_{0})}} u \leqslant
\gamma\,\bigg(|B_{\rho}(x_{0})|^{-1}\int\limits_{B_{\rho}(x_{0})}\,u^{\theta}\,dx\bigg)^{\frac{1}{\theta}} +\gamma\,\rho,
\end{equation}
provided that $B_{16\rho}(x_{0})\subset \Omega$ and
\begin{equation}\label{eq1.12}
\frac{1}{\log\log\dfrac{1}{16\rho}} + \bar{\gamma}_{1}\,L\,\frac{\log\log\dfrac{1}{16\rho}}{\log\dfrac{1}{16\rho}} \leqslant 1.
\end{equation}

Particularly, if $u$ is a bounded non-negative weak solution to equation \eqref{eq1.1}, then
\begin{equation}\label{eq1.13}
\max\limits_{B_{\frac{\rho}{2}(x_{0})}} u \leqslant \gamma \,(\min\limits_{B_{\frac{\rho}{2}}(x_{0})} u +\,\rho),
\end{equation}
provided that $B_{16\rho}(x_{0})\subset \Omega$ and
\begin{equation}\label{eq1.14}
\frac{1}{\log\log\dfrac{1}{16\rho}} + \max(\bar{\gamma}, \bar{\gamma}_{1})\,L\,\frac{\log\log\dfrac{1}{16\rho}}{\log\dfrac{1}{16\rho}} \leqslant 1,
\end{equation}
where $\bar{\gamma} >0$ is the constant defined in Theorem \ref{th1.2}.
\end{theorem}

In the present paper, we substantially refine the results of \cite{AlhutovKrash08, AlkhSurnAlgAn19, ShSkrVoit20, SkrVoitUMB19, SkrVoitNA20, SurnPrepr2018}.
We would like to mention the approach taken in this paper. To prove the interior continuity we use
De\,Giorgi's approach. Let us consider the standard De\,Giorgi's class $DG_{p(\cdot)}(\Omega)$ of functions $u$ which corresponds to
equation \eqref{eq1.1} :
\begin{equation}\label{eq1.15}
\int\limits_{B_{r}(x_{0})}\,|\nabla(u-k)_{\pm}|^{p(x)}\,\zeta^{q}\,dx \leqslant
\gamma \int\limits_{B_{r}(x_{0})}\bigg(\frac{u-k}{r\sigma}\bigg)_{\pm}^{p(x)}\,dx,\quad k\in \mathbb{R}^{1},\quad \sigma\in(0,1),
\end{equation}
$B_{16r}(x_{0})\subset \Omega$ and $\zeta(x)$ is the correspondent cut-off function for the ball $B_{16r}(x_{0})$. Using the Young inequality,
by conditions \eqref{eq1.4} we have
\begin{multline*}
\int\limits_{B_{r}(x_{0})}\,|\nabla(u-k)_{\pm}|^{p_{-}}\,\zeta^{q}\,dx \leqslant
\gamma\,\sigma^{-\gamma}\,\mu(r)\,\int\limits_{B_{r}(x_{0})}\bigg(\frac{u-k}{r}\bigg)_{\pm}^{p_{-}}\,dx +\\
+ \gamma \big|B_{r}(x_{0})\cap\{(u-k)_{\pm}>0 \}\big|, \quad p_{-}:= \min\limits_{B_{r}(x_{0})}p(x).
\end{multline*}
This estimate leads us to condition \eqref{eq1.6} (see, e.g. \cite{SkrVoitUMB19, SkrVoitNA20}). It is easy to see that condition \eqref{eq1.6} fails for the function
$\mu(r)=L\log\log\frac{1}{r}$. To avoid this, using the Young inequality and our choice of $p(x)$ we rewrite inequality
\eqref{eq1.15} as
\begin{multline}\label{eq1.16}
\int\limits_{B_{r}(x_{0})}\,\bigg(\frac{M_{\pm}(k,r)}{r}\bigg)^{p(|x-x_{0}|)}|\nabla(u-k)_{\pm}|^{p}\,\zeta^{q}\,dx \leqslant \\ \leqslant
\gamma\,\sigma^{-\gamma}\,\bigg(\frac{M_{\pm}(k,r)}{r}\bigg)^{p}\int\limits_{B_{r}(x_{0})\cap\{(u-k)_{\pm}>0 \}}\bigg(\frac{M_{\pm}(k,r)}{r}\bigg)^{p(|x-x_{0}|)}\,dx,\quad  M_{\pm}(k,r):=\sup\limits_{B_{r}(x_{0}}(u-k)_{\pm}.
\end{multline}
It appears that the weight $\bigg(\dfrac{M_{\pm}(k,r)}{r}\bigg)^{p(|x-x_{0}|)}$ satisfies the Muckenhoupt type properties. In Section $2$
we define the correspondent weighted De Giorgi's classes by inequalities \eqref{eq1.16} and prove the H\"{o}lder continuity at point $x_{0}$ for the functions which belong to these classes.

The main difficulty arising in the proof of the Harnack type inequalities is related to the so-called theorem on the expansion of positivity.
Roughly speaking, having information on the measure of the "positivity set" of $u$ over the ball
$B_{r}(\overline{x}) \subset B_{\rho}(x_{0}) $:
$$
\left| \left\{ x\in B_{r}(\overline{x}):
u(x)\geqslant m \right\} \right| \geqslant
\alpha(r)\, |B_{r}(\overline{x})|,\quad \alpha(r)=\gamma^{-1} \exp(-\mu^{\beta}(r)),
$$
with some $r>0$, $m>0$ and $\gamma > 1$, and using the standard De\,Giorgi's or
Moser's arguments, we inevitably arrive at  the estimate
$$
u(x)\geqslant \gamma^{-1}\,\exp\big(-\gamma\exp(\mu^{c}(r))\big), \ \ x\in B_{2r}(\overline{x}),
$$
with some $\gamma$, $c >1$. This estimate leads us to condition \eqref{eq1.6}
(see, e.g. \cite{ShSkrVoit20, SkrVoitNA20}). Note that we can not use the classical approach of Krylov and Safonov \cite{KrlvSfnv1980},
DiBenedetto and Trudinger \cite{DiBTru}, as it was done in \cite{BarColMing} under the logarithmic conditions. We also can not use the local clustering lemma of DiBenedetto, Gianazza and Vespri  \cite{DiBGiaVes} (see also \cite{DiBGiaVes1, Sur} ). Difficulties arise not only due to the constant  $\alpha(r)$ which depends on $r$, but also when an additional term, that couldn't be estimated, occurs during the process of iteration from $B_{r}(\bar{x})$ to $B_{\rho}(x_{0})$. To overcome it, we use a workaround that goes back
to Mazya \cite{Maz} and Landis  \cite{Landis_uspehi1963, Landis_mngrph71} papers.

We will demonstrate our approach on the $p$-Laplacian. Fix $x_{0}\in\Omega$ and let $0<r<\rho$,\\
$E\subset B_{r}(x_{0})\subset B_{\rho}(x_{0})$, $B_{16\rho}(x_{0})\subset\Omega$ and consider solution
$v:=v(x,m)$ of the following problem:
\begin{equation}\label{eq1.17}
{\rm div}\big(|\nabla v|^{p-2} \nabla v \big)=0, \quad x\in \mathcal{D}:=B_{16\rho}(x_{0})\setminus E,
\end{equation}
\begin{equation}\label{eq1.18}
v-m\psi\in W^{1,p}_{0}(\mathcal{D}),
\end{equation}
where $m >0$ is some fixed number, and $\psi\in W^{1,p}_{0}(B_{16\rho}(x_{0}))$, $\psi=1$ on $E$.

By the well-known estimate (see e.g. \cite{GarZie}) we have
\begin{equation*}
\min\limits_{B_{4\rho}(x_{0})\setminus B_{2\rho}(x_{0})} v \geqslant \gamma^{-1}\,m\,\bigg(\frac{C_{p}(E)}{\rho^{n-p}}\bigg)^{\frac{1}{p-1}},
\end{equation*}
where $C_{p}(E)$ is a capacity of the set $E$.
By the Poincare inequality from the previous we obtain
\begin{equation}\label{eq1.19}
\min\limits_{B_{4\rho}(x_{0})\setminus B_{2\rho}(x_{0})} v \geqslant \gamma^{-1}\,m\,\bigg(\frac{|E|}{\rho^{n}}\bigg)^{\frac{1}{p-1}},
\end{equation}
Let $u$ be a non-negative bounded super-solution to the $p$-Laplace equation in $\Omega$ and construct the set
$$E(\rho, m):=B_{\rho}(x_{0})\cap\{ u > m \},\quad 0< m < \sup\limits_{\Omega} u. $$
Consider also a solution $v$ of the problem \eqref{eq1.17}, \eqref{eq1.18} with $E$ replaced by $E(\rho, m)$.
Then since $ u\geqslant v$ on $\partial \mathcal{D}$, by the maximal principle and by \eqref{eq1.19} we obtain
\begin{equation*}
m(2\rho):=\min\limits_{ B_{2\rho}(x_{0})} u \geqslant \gamma^{-1}\,m\,\bigg(\frac{|E(\rho, m)|}{\rho^{n}}\bigg)^{\frac{1}{p-1}},
\end{equation*}
which by standard arguments yields for any $\theta\in (0, p-1)$
\begin{multline*}
|B_{\rho}(x_{0})|^{-1}\,\int\limits_{B_{\rho}(x_{0})}\,u^{\theta}\,dx = |B_{\rho}(x_{0})|^{-1}\,\theta\,\int\limits_{0}^{\infty}E(\rho, m)\,m^{\theta-1}\,dm  \leqslant
m^{\theta}(2\rho) + \\ + \gamma m^{p-1}(2\rho)\int\limits_{m(2\rho)}^{\infty}\,m^{\theta-p}\,dm
\leqslant \frac{\gamma}{p-1-\theta}\,m^{\theta}(2\rho),
\end{multline*}
from which the weak Harnack type inequality follows.

In Sections $3,4$ we adapt this simple idea to the case of $p(x)$-Laplacian with non-logarithmic growth.
The weight
$$
\bigg(\frac{M(\rho_{1})}{\rho_{1}}\bigg)^{p(|x-x_{0}|)},\quad M(\rho_{1}):=\sup\limits_{B_{16\rho}(x_{0})\setminus B_{\rho_{1}}(x_{0})}\,v,
\quad \rho <\rho_{1}< 16\rho
$$
which naturally arises in the proof of Theorem \ref{th1.2} also satisfies the Muckenhoupt type conditions.

\begin{remark}
It was unexpected for authors that the modulus of continuity and the constants in the  Harnack type  inequalities do not depend on the additional term $\log\log\dfrac{1}{r}$ (usually, there is a dependency, see e.g. \cite{AlkhSurnAlgAn19, HadSkrVoi, SkrVoitNA20, SurnPrepr2018}).
\end{remark}


The rest of the paper contains the proof of the above theorems.



\section{Elliptic $D G$ classes,
proof of Theorem \ref{th1.1}}\label{Sec2}

In this Section we  define the following De Giorgi's classes.
\begin{definition}
{\rm
We say that a measurable function $u: B_{R}(x_{0})\rightarrow \mathbb{R}$  belongs to the elliptic class
$D G(B_{R}(x_{0}))$ if $u\in W^{1,p}(B_{R}(x_{0}))\cap L^{\infty}(B_{R}(x_{0}))$,
$\esssup\limits_{B_{R}(x_{0})} |u|\leqslant M$ and there exists numbers $1<p<q$, $c_{1}>0$ such that for
any ball $B_{8r}(x_{0})\subset B_{R}(x_{0})$, any $k\in \mathbb{R}$, $|k|<M$, any $\sigma\in(0,1)$,
for any $\zeta\in C_{0}^{\infty}(B_{r}(x_{0}))$,
$0\leqslant\zeta\leqslant1$, $\zeta=1$ in $B_{r(1-\sigma)}(x_{0})$,
$|\nabla \zeta|\leqslant(\sigma r)^{-1}$, the following inequalities hold:
\begin{multline}\label{eq2.1}
\int\limits_{A^{\pm}_{k,r}}
\bigg(\frac{M_{\pm}(u,k,r)}{r}\bigg)^{p(|x-x_{0}|)}
|\nabla u|^{p}\,\zeta^{\,q}\,dx
\leqslant \\ \leqslant
c_{1}\,\sigma^{-q}\,\bigg(\frac{M_{\pm}(u,k,r)}{ r}\bigg)^{p}
\int\limits_{A^{\pm}_{k,r}}
\bigg(\frac{M_{\pm}(u,k,r)}{r} \bigg)^{p(|x-x_{0}|)}dx,
\end{multline}
here $(u-k)_{\pm}:=\max\{\pm(u-k), 0\}$,
$A^{\pm}_{k,r}:=B_{r}(x_{0})\cap \{(u-k)_{\pm}>0\}$,
$M_{\pm}(u,k,r):=\sup\limits_{B_{r}(x_{0})} (u-k)_{\pm}$
and $p(|x-x_{0}|):=L\dfrac{\log\log\frac{1}{|x-x_{0}|}}{\log\frac{1}{|x-x_{0}|}}, \quad L>0$.
}
\end{definition}
We refer to the parameters $c_{1}$, $n$, $p$, $q$ and $M$ as our structural
data, and we write $\gamma$ if it can be quantitatively determined a priory in terms of the above
quantities. The generic constant $\gamma$ may change from line to line.

Our main result of this Section reads as follows:
\begin{theorem}\label{th2.1}
Let $u\in D G(B_{R}(x_{0}))$, then $u$ is H\"{o}lder continuous at $x_{0}$.
\end{theorem}

We note that the solutions of equation \eqref{eq1.1} belong to the corresponding
$D G(B_{R}(x_{0}))$ classes, provided that $B_{2R}(x_{0}) \subset \Omega$.
We test identity \eqref{eq1.7} by $\varphi=(u-k)_{\pm}\zeta^{q}(x)$, by the Young inequality we obtain
\begin{equation*}
\int\limits_{A^{\pm}_{k,r}}|\nabla u|^{p(x)}\,dx \leqslant \gamma \,\int\limits_{A^{\pm}_{k,r}}\bigg(\frac{u-k}{\sigma\,r}\bigg)_{\pm}^{p(x)}\,dx\leqslant \gamma \sigma^{-\gamma}\bigg(\frac{M_{\pm}(u,k,r)}{ r}\bigg)^{p}
\int\limits_{A^{\pm}_{k,r}}\bigg(\frac{M_{\pm}(u,k,r)}{r} \bigg)^{p(|x-x_{0}|)}dx.
\end{equation*}
From this, using again the Young inequality
\begin{equation*}
\int\limits_{A^{\pm}_{k,r}}\bigg(\frac{M_{\pm}(u,k,r)}{r}\bigg)^{p(|x-x_{0}|)}|\nabla u|^{p}\,\zeta^{\,q}\,dx \leqslant
\int\limits_{A^{\pm}_{k,r}}|\nabla u|^{p(x)}\,dx +\int\limits_{A^{\pm}_{k,r}}\bigg(\frac{M_{\pm}(u,k,r)}{r}\bigg)^{p(x)}\,dx,
\end{equation*}
from which the required \eqref{eq2.1} follows.

\subsection{ Auxiliary Propositions}

For $k\in \mathbb{R}$ and $0<r< R$ set $w_{\pm}(x,u,k,r) := \bigg(\dfrac{M_{\pm}(u,k,r)}{r}\bigg)^{p(|x-x_{0}|)}$, further we need the following lemmas
\begin{lemma}\label{lem2.1}
There exists $C>0$ depending only on the data, such that for any $u\in D G(B_{R}(x_{0}))$ and for any $t>0$ the following inequalities hold
\begin{equation}\label{eq2.2}
r^{-n}\,\int\limits_{B_{r}(x_{0})} w_{\pm}(x,u,k,r)\,dx \,\,\,\bigg(r^{-n}\,\int\limits_{B_{r}(x_{0})}w_{\pm}^{-t}(x,u,k,r)\,dx\bigg)^{\frac{1}{t}}\leqslant \gamma^{1+\frac{1}{t}},
\end{equation}
\begin{equation}\label{eq2.3}
\bigg(r^{-n}\,\int\limits_{B_{r}(x_{0})}w_{\pm}^{1+t}(x,u,k,r)\,dx\bigg)^{\frac{1}{1+t}} \leqslant \gamma^{\frac{1}{1+t}+1}\,r^{-n}\,\int\limits_{B_{r}(x_{0})} w_{\pm}(x,u,k,r)\,dx,
\end{equation}
provided that
\begin{equation}\label{eq2.4}
\frac{1}{\log\log\frac{1}{r}} + t\,C\,L\,\frac{\log\log\frac{1}{r}}{\log\frac{1}{r}}\leqslant 1,\quad \text{and}\quad r\leqslant M_{\pm}(u,k,r) \leqslant 1.
\end{equation}
\end{lemma}
\begin{proof}
To prove inequalities \eqref{eq2.2}, \eqref{eq2.3} we just need to check
\begin{equation}\label{eq2.5}
\gamma^{-1}\bigg(\frac{M_{\pm}(u,k,r)}{r}\bigg)^{-tp(r)}\leqslant r^{-n}\,\int\limits_{B_{r}(x_{0})} w_{\pm}^{-t}(x,u, k, r)\,dx\leqslant
\gamma \bigg(\frac{M_{\pm}(u,k,r)}{r}\bigg)^{-tp(r)},\quad  t>0,
\end{equation}
\begin{equation}\label{eq2.6}
\gamma^{-1}\bigg(\frac{M_{\pm}(u,k,r)}{r}\bigg)^{tp(r)}\leqslant r^{-n}\,\int\limits_{B_{r}(x_{0})} w_{\pm}^{t}(x,u, k, r)\,dx\leqslant
\gamma \bigg(\frac{M_{\pm}(u,k,r)}{r}\bigg)^{tp(r)},\quad  t>0.
\end{equation}
The left inequality in \eqref{eq2.5} and the right inequality in \eqref{eq2.6} are obvious due to the fact that $p(|x-x_{0}|)$ is increasing
if $x\in B_{r}(x_{0})$ and $r$ is sufficiently small. Let us check the right inequality in \eqref{eq2.5}. Integrating by parts, using the fact that $\dfrac{\log\log\frac{1}{s}}{\log^{2}\frac{1}{s}}$ is increasing on the interval $(0,r)$, we obtain
\begin{multline*}
\int\limits_{B_{r}(x_{0})} w_{\pm}^{-t}(x,u, k, r)\,dx= \gamma \int\limits_{0}^{r}\bigg(\frac{M_{\pm}(u,k,r)}{r}\bigg)^{-tp(s)}\,s^{n-1}\,ds
\leqslant\gamma r^{n}\,\bigg(\frac{M_{\pm}(u,k,r)}{r}\bigg)^{-tp(r)} +\\+\gamma t\,L\, \log\frac{M_{\pm}(u,k,r)}{r}\,\int\limits_{0}^{r}\bigg(\frac{M_{\pm}(u,k,r)}{r}\bigg)^{-tp(s)}\,\frac{\log\log\frac{1}{s}}{\log^{2}\frac{1}{s}}\,s^{n-1}\,ds \leqslant \\ \leqslant
\gamma r^{n}\,\bigg(\frac{M_{\pm}(u,k,r)}{r}\bigg)^{-tp(r)} +\gamma\,t\,L\,\log\frac{M_{\pm}(u,k,r)}{r}\frac{\log\log\frac{1}{r}}{\log^{2}\frac{1}{r}}\int\limits_{0}^{r}\bigg(\frac{M_{\pm}(u,k,r)}{r}\bigg)^{-tp(s)}\,s^{n-1}\,ds \leqslant \\ \leqslant
\gamma r^{n}\,\bigg(\frac{M_{\pm}(u,k,r)}{r}\bigg)^{-tp(r)} +\frac{1}{2}\,\int\limits_{B_{r}(x_{0})} w_{\pm}^{-t}(x,u,k, r)\,dx,
\end{multline*}
provided that $\gamma\,t\,L\, \dfrac{\log\log\frac{1}{r}}{\log\frac{1}{r}}\leqslant \dfrac{1}{2},$ $r \leqslant M_{\pm}(u, k, r) \leqslant1$, from which the required inequality follows.

Similarly,
\begin{multline*}
\int\limits_{B_{r}(x_{0})} w_{\pm}^{t}(x,u, k, r)\,dx \geqslant \gamma r^{n}\,\bigg(\frac{M_{\pm}(u,k,r)}{r}\bigg)^{tp(r)}- \\ -\gamma t\,L\, \log\frac{M_{\pm}(u,k,r)}{r}\,\int\limits_{0}^{r}\bigg(\frac{M_{\pm}(u,k,r)}{r}\bigg)^{tp(s)}\,\frac{\log\log\frac{1}{s}}{\log^{2}\frac{1}{s}}\,s^{n-1}\,ds
\geqslant\\ \geqslant \gamma r^{n}\,\bigg(\frac{M_{\pm}(u,k,r)}{r}\bigg)^{tp(r)}\bigg(1-\gamma t\,L\,\dfrac{\log\log\frac{1}{r}}{\log\frac{1}{r}}\bigg),
\end{multline*}
from which the left inequality in \eqref{eq2.6} follows, provided that  $\gamma\,t\,L\, \dfrac{\log\log\frac{1}{r}}{\log\frac{1}{r}}\leqslant \dfrac{1}{2},$ \\$r\leqslant M_{\pm}(u, k, r)\leqslant1$, which completes the proof of  the lemma.
\end{proof}

In the sequel we also need the following lemma
\begin{lemma}\label{lem2.2}
There exist $C_{1}>0$, $\kappa_{1} >1 $ such that for any $u\in D G(B_{R}(x_{0}))$ and any \\$\varphi \in W_{0}(B_{r}(x_{0}))$ the following inequality holds
\begin{multline}\label{eq2.7}
\frac{1}{w^{\pm}_{u,k,r}(B_{r}(x_{0}))} \int\limits_{B_{r}(x_{0})} w_{\pm}(x,u,k,r) |\varphi|^{\kappa_{1} p} \,dx \leqslant \\
\leqslant \gamma \bigg(r^{p}\,\frac{1}{w^{\pm}_{u,k,r}(B_{r}(x_{0}))} \int\limits_{B_{r}(x_{0})} w_{\pm}(x,u,k,r) |\nabla \varphi|^{p}\,dx \bigg)^{\kappa_{1}} ,
\end{multline}
provided that
\begin{equation}\label{eq2.8}
\frac{1}{\log\log\frac{1}{r}} + C_{1}\,L\,\frac{\log\log\frac{1}{r}}{\log\frac{1}{r}}\leqslant 1,\quad \text{and}\quad r\leqslant M_{\pm}(u,k,r) \leqslant 1.
\end{equation}
Here
$$w^{\pm}_{u,k,r}(F) := \int\limits_{F}w_{\pm}(x,u,k,r)\,dx ,\quad F\subset \mathbb{R}^{n}.$$
\end{lemma}
\begin{proof}
Inequality \eqref{eq2.7} is a consequence of \eqref{eq2.5}, \eqref{eq2.6} and Sobolev embedding theorem. Indeed, using the H\"{o}lder inequality, if $0 < \delta < \dfrac{p}{p+n}$ and $1 <\kappa_{1} <\dfrac{n(1-\delta)}{n-p(1-\delta)}$, we obtain with $t=\dfrac{n(1-\delta)}{n(1-\delta)-\kappa_{1}(n-p(1-\delta))}>1$
\begin{multline*}
\int\limits_{B_{r}(x_{0})} w_{\pm}(x,u,k,r) |\varphi|^{\kappa_{1} p} \,dx \leqslant \bigg(\int\limits_{B_{r}(x_{0})} w_{\pm}^{t}(x,u,k,r)\,dx\bigg)^{\frac{1}{t}}\bigg(\int\limits_{B_{r}(x_{0})} |\varphi|^{\frac{np(1-\delta)}{n-p(1-\delta)}}\,dx\bigg)^{\kappa_{1}\frac{n-p(1-\delta)}{n(1-\delta)}}\leqslant \\ \leqslant \gamma \bigg(\int\limits_{B_{r}(x_{0})} w_{\pm}^{t}(x,u,k,r)\,dx\bigg)^{\frac{1}{t}}\bigg(\int\limits_{B_{r}(x_{0})} |\nabla \varphi|^{p(1-\delta)}\,dx\bigg)^{\frac{\kappa_{1}}{1-\delta}}\leqslant \\ \leqslant \gamma\bigg(\int\limits_{B_{r}(x_{0})} w_{\pm}^{t}(x,u,k,r)\,dx\bigg)^{\frac{1}{t}} \bigg(\int\limits_{B_{r}(x_{0})} w_{\pm}^{-\frac{1-\delta}{\delta}}(x,u,k,r)\,dx\bigg)^{\frac{\kappa_{1}\delta}{1-\delta}}\bigg(\int\limits_{B_{r}(x_{0})} w_{\pm}(x,u,k,r) |\nabla \varphi|^{p}\,dx\bigg)^{\kappa_{1}}\\ \leqslant \gamma r^{p\kappa_{1}}
\bigg(\int\limits_{B_{r}(x_{0})} w_{\pm}(x,u,k,r)\,dx\bigg)^{1-\kappa_{1}}\,\bigg(\int\limits_{B_{r}(x_{0})} w_{\pm}(x,u,k,r) |\nabla \varphi|^{p}\,dx\bigg)^{\kappa_{1}}.
\end{multline*}
Choosing $C_{1}=t\,C$, we arrive at the required \eqref{eq2.7}, which completes the proof of the lemma.
\end{proof}

\subsection{De Giorgi Type Lemma}
Let $B_{8r}(x_{0})\subset B_{R}(x_{0})$ and let
$\mu^{+}_{r}\geqslant \esssup\limits_{B_{r}(x_{0})}u, \quad
\mu^{-}_{r}\leqslant \essinf\limits_{B_{r}(x_{0})}u, \quad
\omega_{r}:=\mu^{+}_{r}-\mu^{-}_{r}$.

\begin{lemma}\label{lem2.3}
Let $u\in D G(B_{R}(x_{0}))$ and fix $\xi\in(0,\dfrac{1}{2M})$.
Then there exists $\nu\in(0,1)$ depending only on
$n$, $p$, $q$, $c_{1}$ and $M$, such that if
\begin{equation}\label{eq2.9}
w^{+}_{u,\mu_r^+-\xi\omega_{r},r}\bigg(B_{r}(x_{0}):u \geqslant \mu_r^+-\xi\omega_{r}\bigg) \leqslant \nu\,w^{+}_{u,\mu_r^+-\xi\omega_{r},r}(B_{r}(x_{0})),
\end{equation}
then either
\begin{equation}\label{eq2.10}
\xi\,\omega_{r}\leqslant 4\,r,
\end{equation}
or
\begin{equation}\label{eq2.11}
u(x)\leqslant \mu_r^+ -\frac{\xi}{4}\,\omega_{r}
\quad \text{for a.a.} \ x\in B_{\frac{r}{2}}(x_{0}),
\end{equation}
provided that
\begin{equation}\label{eq2.12}
\frac{1}{\log\log\frac{1}{r}} + C_{1}\,L\,\frac{\log\log\frac{1}{r}}{\log\frac{1}{r}}\leqslant 1,
\end{equation}
where $C_{1}$ is the constant defined in Lemma \ref{lem2.2}.

Likewise, if
\begin{equation}\label{eq2.13}
w^{-}_{u,\mu_r^{-}+\xi\omega_{r},r}\bigg(B_{r}(x_{0}):u \leqslant \mu_r^{-}+\xi\omega_{r}\bigg) \leqslant \nu\,w^{-}_{u,\mu_r^{-}+\xi\omega_{r},r}(B_{r}(x_{0})),
\end{equation}
then either \eqref{eq2.10} holds, or
\begin{equation}\label{eq2.14}
u(x)\geqslant \mu_r^{-} +\frac{\xi}{4}\,\omega_{r}
\quad \text{for a.a.} \ x\in B_{\frac{r}{2}}(x_{0}),
\end{equation}
provided that  \eqref{eq2.12} is valid.
\end{lemma}
\begin{proof}
We provide the proof of \eqref{eq2.11}, while the proof of \eqref{eq2.14}
is completely similar.
For $j=0,1,2,\ldots$ we set $r_{j}:=\dfrac{r}{2}(1+2^{-j})$,
$k_{j}:=\mu^{+}_{r}- \dfrac{\xi}{2}\,\omega_{r}-
\xi\,\omega_{r}2^{-j-1}$,  let $\zeta_{j}(x)\in C^{\infty}_{0}(B_{r_{j}}(x_{0})),\\ 0\leqslant \zeta_{j}(x)\leqslant 1,
\zeta_{j}(x)=1$ for $x\in B_{r_{j+1}}(x_{0})$ and set $A_{j}:=B_{r_{j}}(x_{0})\cap\{u \geqslant k_{j}\}$. Further we will assume that $\sup\limits_{B_{\frac{r}{2}}(x_{0})}(u-k_{\infty})_{+} \geqslant \dfrac{\xi}{4}
\omega_{r}$, because otherwise inequality \eqref{eq2.11} is evident. We note that
\begin{equation*}
\gamma^{-1}\,w_{+}(x,u,k_{j},r_{j})\leqslant  w_{+}(x,u,\mu_r^+-\xi\omega_{r},r)\leqslant \gamma w_{+}(x,u,k_{j},r_{j}), \quad x\in B_{r_{j}}(x_{0}).
\end{equation*}
If \eqref{eq2.10} is violated, then condition \eqref{eq2.8} holds due to \eqref{eq2.12} and the choice of $\xi$, \\$r_{j}\leqslant \sup\limits_{B_{r_{j}}}(u-k_{j})_{+}\leqslant 1$. So, by Lemma \ref{lem2.2} and inequality \eqref{eq2.1} we have
\begin{multline*}
(k_{j} - k_{j+1})^{p}\,w^{+}_{u,\mu_r^+-\xi\omega_{r},r}( A_{j+1})\leqslant \gamma  2^{j\gamma} \, \int\limits_{B_{r_{j}}(x_{0})} w_{+}(x,u, k_{j}, r_{j})\,(u-k_{j})^{p}_{+}\zeta^{q}_{j}\,dx\leqslant\\ \leqslant \,\gamma 2^{j\gamma}\bigg(\int\limits_{B_{r_{j}}(x_{0})} w_{+}(x,u,k_{j}, r_{j})\,\big((u-k_{j})^{p}_{+}\zeta^{q}_{j}\big)^{\kappa_{1}}\,dx\bigg)^{\frac{1}{\kappa_{1}}} \big[w^{+}_{u,\mu_r^+-\xi\omega_{r},r}(A_{j})\big]^{1-\frac{1}{\kappa_{1}}} \leqslant \\ \leqslant \gamma\,2^{j\gamma} \big[w^{+}_{u,\mu_r^+-\xi\omega_{r},r}(B_{r}(x_{0}))\big]^{\frac{1}{\kappa_{1}}-1}\,r^{p}\,\int\limits_{B_{r_{j}}(x_{0})} w_{+}(x,u, k_{j},r_{j}) |\nabla\big((u-k_{j})_{+}\zeta^{q}_{j}\big)|^{p}\,dx\,\times\\ \times\big[w^{+}_{u,\mu_r^+-\xi\omega_{r},r}(A_{j})\big]^{1-\frac{1}{\kappa_{1}}} \leqslant \gamma \, 2^{j\gamma}\,\big(\xi\omega_{r}\big)^{p}\,\big[w^{+}_{u,\mu_r^+-\xi\omega_{r},r}(B_{r}(x_{0}))\big]^{\frac{1}{\kappa_{1}}-1}\,\big[w^{+}_{u,\mu_r^+-\xi\omega_{r},r}(A_{j})\big]^{2-\frac{1}{\kappa_{1}}},
\end{multline*}
which implies
\begin{equation*}
y_{j+1}:=\frac{w^{+}_{u,\mu_r^+-\xi\omega_{r},r}(A_{j+1})}{w^{+}_{u,\mu_r^+-\xi\omega_{r},r}(B_{r}(x_{0}))}\leqslant \gamma 2^{j\gamma} y_{j}^{2-\frac{1}{\kappa_{1}}},
\end{equation*}
from which by standard arguments (see e.g. \cite{LadUr}) the required \eqref{eq2.11} follows, provided that $\nu$ is chosen to satisfy $\nu\leqslant \gamma^{-1}$.
This completes the proof of the lemma.
\end{proof}
\subsection{Expansion of the Positivity}
To prove our next result we need the following lemma.
\begin{lemma}\label{lem2.4}
Let  $k < l$, $0< \delta <1-\frac{1}{p}$, $u \in D G(B_{R}(x_{0}))$, $\varphi\in W^{1,p(1-\delta)}(B_{r}(x_{0}))$, then
\begin{equation}\label{eq2.15}
(l-k)\,\frac{w^{+}_{u,l,r}(A^{+}_{l,r})}{w^{+}_{u,l,r}(B_{r}(x_{0}))}\,\, \frac{|A^{-}_{k,r}|}{|B_{r}(x_{0})|}\leqslant
\gamma r^{1-\frac{n}{p(1-\delta)}}\bigg(\int\limits_{A^{+}_{k,r}\setminus A^{+}_{l,r}}|\nabla \varphi|^{p(1-\delta)}\,dx\bigg)^{\frac{1}{p(1-\delta)}},
\end{equation}
\begin{equation}\label{eq2.16}
(l-k)\,\frac{w^{-}_{u,k,r}(A^{-}_{k,r})}{w^{-}_{u,k,r}(B_{r}(x_{0}))}\,\, \frac{|A^{+}_{l,r}|}{|B_{r}(x_{0})|}\leqslant
\gamma r^{1-\frac{n}{p(1-\delta)}}\bigg(\int\limits_{A^{-}_{l,r}\setminus A^{-}_{k,r}}|\nabla \varphi|^{p(1-\delta)}\,dx\bigg)^{\frac{1}{p(1-\delta)}},
\end{equation}
provided that
\begin{equation}\label{2.17}
\frac{1}{\log\log\frac{1}{r}} + \frac{L\,C}{p(1-\delta)-1}\,\frac{\log\log\frac{1}{r}}{\log\frac{1}{r}}\leqslant 1,\quad \text{and}\quad r\leqslant M_{+}(u,l,r), M_{-}(u,k,r)\leqslant 1,
\end{equation}
here $C>1$ is the constant, defined in Lemma \ref{lem2.1}.
\end{lemma}
\begin{proof}
Let $\{v\}_{r}= |B_{r}(x_{0})|^{-1} \int\limits_{B_{r}(x_{0})} v\,dx$. Using the Poincare inequality and inequality \eqref{eq2.3} with $t=\frac{1}{p(1-\delta)-1}$ we get
\begin{multline*}
\int\limits_{B_{r}(x_{0})} w_{+}(x,u,l,r) |v-\{v\}_{r}| \,dx \leqslant \gamma \bigg(\int\limits_{B_{r}(x_{0})} w_{+}^{\frac{p(1-\delta)}{p(1-\delta)-1}}(x,u,l,r)\,dx\bigg)^{1-\frac{1}{p(1-\delta)}}\times\\ \times \bigg(\int\limits_{B_{r}(x_{0})}  |v-\{v\}_{r}|^{p(1-\delta)} \,dx\bigg)^{\frac{1}{p(1-\delta)}} \leqslant\gamma r^{1-\frac{n}{p(1-\delta)}}\,w^{+}_{u,l,r}(B_{r}(x_{0}))\,\bigg( \int\limits_{B_{r}(x_{0})} |\nabla v|^{p(1-\delta)}\,dx\bigg)^{\frac{1}{p(1-\delta)}}.
\end{multline*}
Take $v=0$, if $\varphi < k$, $v= \varphi-k$, if $k < \varphi < l$, $v= l -k$, if $\varphi > l$. We evidently have
$\{v\}_{r} \leqslant (l-k) \dfrac{|A^{+}_{k,r}|}{|B_{r}(x_{0})|}$, hence
\begin{equation*}
\int\limits_{B_{r}(x_{0})} w_{+}(x,u,l,r) |v-\{v\}_{r}| \,dx \geqslant (l-k) \bigg(1-\frac{|A^{+}_{k,r}|}{|B_{r}(x_{0})|}\bigg)\int\limits_{A^{+}_{l,r}} w_{+}(x,u,l,r) \,dx,
\end{equation*}
from which the required inequality \eqref{eq2.15} follows. The proof of \eqref{eq2.16} is completely similar.
\end{proof}

\begin{lemma}\label{lem2.5}
\textbf{(Expansion of the Positivity)} Let $u\in D G(B_{R}(x_{0}))$, fix $\xi \in(0,\frac{1}{2M})$ and
assume that with some $\alpha\in(0,1)$ there holds
\begin{equation}\label{eq2.18}
\left| \left\{ x\in B_{r}(x_{0}):u(x)
\geqslant \mu_{+} -\xi\,\omega_{r} \right\} \right|
\leqslant(1-\alpha)\, |B_{r}(x_{0})|.
\end{equation}
Then there exists number $s_{*}$ depending only on
$n$, $p$, $q$, $c_{1}$, $M$, $\alpha$ and $\xi$ such that either
\begin{equation}\label{eq2.19}
\omega_{r}\leqslant  2^{s_{*}+1}\,r ,
\end{equation}
or
\begin{equation}\label{eq2.20}
u(x)\leqslant \mu_{+}-2^{-s_{*}-1}\,\omega_{r}
\quad \text{for a.a.} \ x\in B_{\frac{r}{2}}(x_{0}),
\end{equation}
provided that
\begin{equation}\label{eq2.21}
\frac{1}{\log\log\frac{1}{r}} + L\,(s_{*}+1)\,\frac{\log\log\frac{1}{r}}{\log\frac{1}{r}}\leqslant 1.
\end{equation}
Likewise, if
\begin{equation}\label{eq2.22}
\left| \left\{ x\in B_{r}(x_{0}):u(x)
\leqslant \mu_{-} +\xi\,\omega_{r} \right\} \right|
\leqslant(1-\alpha)\, |B_{r}(x_{0})|,
\end{equation}
then there exists number $s_{*}$ depending only on
$n$, $p$, $q$, $c_{1}$, $M$, $\alpha$ and $\xi$, such that either \eqref{eq2.19} holds or
\begin{equation}\label{eq2.23}
u(x)\geqslant \mu_{-}+2^{-s_{*}-1}\,\omega_{r}
\quad \text{for a.a.} \ x\in B_{\frac{r}{2}}(x_{0}),
\end{equation}
provided that \eqref{eq2.21} holds.
\end{lemma}
\begin{proof}
We provide the proof of \eqref{eq2.20},
while the proof of \eqref{eq2.23} is completely similar.
We set $k_{s}:=\mu_{r}^{+}-\dfrac{\omega_{r}}{2^{s}}$,
$s=\log\dfrac{1}{\xi}, \ldots,s_{*}-1$, where $s_{*}$ is large enough to be chosen later.
We will assume that $\sup\limits_{B_{\frac{r}{2}}(x_{0})}(u-k_{s_{*}})_{+}\geqslant \dfrac{\omega_{r}}{2^{s_{*}+1}}$,  since otherwise inequality \eqref{eq2.20} is evident. If inequality \eqref{eq2.19} is violated, then Lemma \ref{lem2.5} with $l=k_{s+1}$ and $k=k_{s}$ yields
\begin{multline*}
\frac{\omega_{r}}{2^{s+1}}\frac{w^{+}_{u,k_{s+1},r}(A^{+}_{k_{s+1},r})}{w^{+}_{u,k_{s+1},r}(B_{r}(x_{0}))}\leqslant
\gamma(\alpha)\, r^{1-\frac{n}{p(1-\delta)}}\bigg(\int\limits_{A^{+}_{k_{s},r}\setminus A^{+}_{k_{s+1},r}}|\nabla u|^{p(1-\delta)}\,dx\bigg)^{\frac{1}{p(1-\delta)}}\leqslant \\ \leqslant \gamma(\alpha)\,r^{1-\frac{n}{p(1-\delta)}}\,\bigg(\int\limits_{A^{+}_{k_{s},r}\setminus A^{+}_{k_{s+1},r}}w_{+}^{-\frac{1-\delta}{\delta}}(x,u,k_{s},r)dx\bigg)^{\frac{\delta}{p(1-\delta)}}\bigg(\int\limits_{A^{+}_{k_{s},r}}w_{+}(x,u,k_{s},r)|\nabla u|^{p}\,dx\bigg)^{\frac{1}{p}}.
\end{multline*}
From this, by inequality \eqref{eq2.1} we obtain
\begin{equation*}
\frac{w^{+}_{u,k_{s+1},r}(A^{+}_{k_{s+1},r})}{w^{+}_{u,k_{s+1},r}(B_{r}(x_{0}))}\leqslant \gamma(\alpha)\,r^{-\frac{n}{p(1-\delta)}}\bigg(\int\limits_{A^{+}_{k_{s},r}\setminus A^{+}_{k_{s+1},r}}w_{+}^{-\frac{1-\delta}{\delta}}(x,u,k_{s},r)dx\bigg)^{\frac{\delta}{p(1-\delta)}}[w^{+}_{u,k_{s},r}(B_{2r}(x_{0}))]^{\frac{1}{p}}.
\end{equation*}
By our choice and  \eqref{eq2.21} we have
\begin{equation*}
\bigg(\frac{M_{+}(u, k_{s}, r)}{M_{+}(u, k_{s_{*}}, r)}\bigg)^{p(|x-x_{0}|)} \leqslant 2^{(s_{*}+1-s)p(|x-x_{0}|)}
\leqslant 2^{L (s_{*}+1)\frac{\log\log\frac{1}{r}}{\log\frac{1}{r}}}\leqslant 2,\quad x\in B_{r}(x_{0}),
\end{equation*}
therefore  $w^{+}_{u,k_{s+1},r}(A^{+}_{k_{s+1},r})\geqslant \gamma^{-1} w^{+}_{u,k_{s_{*}},r}(A^{+}_{k_{s_{*}},r})$
and $w^{+}_{u,k_{s+1},r}(B_{r}(x_{0})) + w^{+}_{u,k_{s},r}(B_{2r}(x_{0}))\leqslant \gamma w^{+}_{u,k_{s_{*}},r}(B_{r}(x_{0}))$, $s=\log\dfrac{1}{\xi}, \ldots,s_{*}-1$, so from the previous relation we have
\begin{equation*}
[w^{+}_{u,k_{s_{*}},r}(A^{+}_{k_{s_{*}},r})]^{\frac{p(1-\delta)}{\delta}}\leqslant \gamma(\alpha)r^{-\frac{n}{\delta}}\,[w^{+}_{u,k_{s_{*}},r}(B_{r}(x_{0}))]^{(p+1)\frac{1-\delta}{\delta}}\,\int\limits_{A^{+}_{k_{s},r}\setminus A^{+}_{k_{s+1},r}}w_{+}^{-\frac{1-\delta}{\delta}}(x,u,k_{s_{*}},r)dx.
\end{equation*}
Summing up these inequalities over $s=\log\dfrac{1}{\xi}, \ldots,s_{*}-1$, we conclude that
\begin{multline*}
(s_{*}-\log\frac{1}{\xi}-1)[w^{+}_{u,k_{s_{*}},r}(A^{+}_{k_{s_{*}},r})]^{\frac{p(1-\delta)}{\delta}}\leqslant \\ \leqslant \gamma(\alpha)r^{-\frac{n}{\delta}}\,[w^{+}_{u,k_{s_{*}},r}(B_{r}(x_{0}))]^{(p+1)\frac{1-\delta}{\delta}}\,\int\limits_{B_{r}(x_{0})}w_{+}^{-\frac{1-\delta}{\delta}}(x,u,k_{s_{*}},r)dx.
\end{multline*}
Using inequality \eqref{eq2.2} from the last inequality we arrive at
\begin{equation*}
w^{+}_{u,k_{s_{*}},r}(A^{+}_{k_{s_{*}},r})\leqslant \gamma(\alpha)\bigg(s_{*}-\log\frac{1}{\xi}-1\bigg)^{-\frac{\delta}{p(1-\delta)}}
w^{+}_{u,k_{s_{*}},r}(B_{r}(x_{0})).
\end{equation*}
Choosing $s_{*}$ by the condition $\gamma(\alpha)\bigg(s_{*}-\log\frac{1}{\xi}-1\bigg)^{-\frac{\delta}{p(1-\delta)}}=\nu$ and
using  Lemma \ref{lem2.3} we obtain  \eqref{eq2.20}, which proves Lemma \ref{lem2.5}.
\end{proof}

\subsection{Proof of Theorems \ref{th1.1}, \ref{th2.1}}
To complete the proof of Theorems \ref{th1.1} and \ref{th2.1} we fix $R$ by the condition
\begin{equation*}
\frac{1}{\log\log\frac{1}{R}} + L\,(s_{*}+1)\,\frac{\log\log\frac{1}{R}}{\log\frac{1}{R}}\leqslant 1,
\end{equation*}
where $s_{*}$ is the number defined in Lemma \ref{lem2.5}, and  assume that the following two alternative
cases are possible:
$$
\left|\left\{x\in B_{r}(x_{0}):
u(x)\geqslant \mu^{+} -\frac{\omega_{r}}{2^{s_{0}}}\right\} \right|
\leqslant \frac{1}{2}\,|B_{r}(x_{0})|,\quad s_{0}\geqslant 2 +[\log M],
$$
or
$$
\left|\left\{x\in B_{r}(x_{0}):
u(x)\leqslant \mu^{-} +\frac{\omega_{r}}{2^{s_{0}}}\right\} \right|
\leqslant \frac{1}{2}\,|B_{r}(x_{0})|
$$
for any $0<r< \rho < R$. Assume, for example, the first one holds. Then by Lemma \ref{lem2.5}  we obtain
$$
\omega_{\frac{r}{2}}\leqslant \left(1- 2^{-s_{*}-1}\right)
\omega_{r}+ 2^{s_{*}+1} r.
$$
Iterating this inequality, we have
\begin{equation*}
\omega_{r}\leqslant \gamma\,M\, \bigg(\frac{r}{\rho}\bigg)^{\beta} +\gamma\,\rho, \quad \beta=\beta(s_{*})\in (0,1).
\end{equation*}
This completes the proof of Theorems \ref{th1.1} and \ref{th2.1}.


\section{Upper and lower estimates of auxiliary solutions}\label{Sec3}

In this Section we prove upper and lower bounds for auxiliary solutions
$v:=v(x,m)$ to the problem
\begin{equation*}
{\rm div} \mathbf{A}(x, \nabla v) =0, \quad x\in \mathcal{D}:=B_{16\rho}(x_{0})\setminus E,\quad E\subset B_{\rho}(x_{0}),
\end{equation*}
\begin{equation*}
v-m\psi\in W_{0}(\mathcal{D}),
\end{equation*}
where $0 <m \leqslant M$ is some fixed number, and $\psi\in W_{0}(B_{16\rho}(x_{0}))$, $\psi=1$ on $E$.
The existence of the solutions
$v$ follows from the general theory of monotone operators. We will assume that the following
integral identity holds:
\begin{equation}\label{eq3.1}
\int_{\mathcal{D}}\mathbf{A}(x, \nabla v)\,\nabla\varphi \,dx=0
\quad \text{for any } \ \varphi\in W_{0}(\mathcal{D}).
\end{equation}

Testing \eqref{eq3.1} by $\varphi=(v-m)_{+}$ and by $\varphi=v_{-}$ and using condition
\eqref{eq1.8}, we obtain that $0\leqslant v\leqslant m\leqslant M$.

For $\rho<\rho_{1}<\rho_{2}\leqslant 16\rho$ we set $K(\rho_{1},\rho_{2}):=B_{\rho_{2}}(x_{0})\setminus B_{\rho_{1}}(x_{0})$,
$M(\rho_{1}):=\sup\limits_{K(\rho_{1}, 16\rho)} v$, $w(x,\rho_{1}):=\bigg(1+\dfrac{M(\rho_{1})}{\rho_{1}}\bigg)^{p(|x-x_{0}|)}$,
$w(\rho_{1},\rho_{2}):=\int\limits_{K(\rho_{1},\rho_{2})}w(x,\rho_{1})\,dx.$

Note that similarly to \eqref{eq2.5}, \eqref{eq2.6}, for all $\rho_{1}\in(\rho, 16\rho)$ there hold
\begin{equation}\label{eq3.2}
\gamma^{-1}\bigg(1+\frac{M(\rho_{1})}{\rho_{1}}\bigg)^{-tp(16\rho)}\leqslant |K(\rho_{1},16\rho)|^{-1}\,\int\limits_{K(\rho_{1},16\rho)} w^{-t}(x,\rho_{1})\,dx\leqslant \gamma \bigg(1+\frac{M(\rho_{1})}{\rho_{1}}\bigg)^{-tp(16\rho)},  t>0,
\end{equation}
\begin{equation}\label{eq3.3}
\gamma^{-1}\bigg(1+\frac{M(\rho_{1})}{\rho_{1}}\bigg)^{tp(16\rho)}\leqslant |K(\rho_{1},16\rho)|^{-1}\,\int\limits_{K(\rho_{1},16\rho)} w^{t}(x,\rho_{1})\,dx\leqslant \gamma \bigg(1+\frac{M(\rho_{1})}{\rho_{1}}\bigg)^{tp(16\rho)},\quad  t>0,
\end{equation}
provided that
\begin{equation}\label{eq3.4}
\frac{1}{\log\log\dfrac{1}{16\rho}} + t\,\bar{C}\,L\,\log(1+\frac{M}{\rho})\frac{\log\log\dfrac{1}{16\rho}}{\log^{2}\dfrac{1}{16\rho}} \leqslant 1,
\end{equation}
where $\bar{C}=\max(C, C_{1})$ and $C$, $C_{1}$ are the constants defined in  Lemmas  \ref{lem2.1}, \ref{lem2.2}. Therefore Lemmas \ref{lem2.1}, \ref{lem2.2} continue to hold in $K(\rho_{1},16\rho)$ with $w_{\pm}(x,u,k,r)$ replaced by $w(x,\rho_{1})$.

To formulate our  results, we need the notion of the capacity.
Let $E\subset B_{r}(x_{0})\subset  B_{\rho}(x_{0})$
and for any $m>0$ set
$$
C_{p(\cdot)}(E, B_{16\rho}(x_{0});m):=\frac{1}{m}\,\inf\limits_{\varphi\in \mathfrak{M}(E)}
\int\limits_{B_{16\rho}(x_{0})}|m\nabla \varphi|^{p(x)}\,dx,
$$
where the infimum is taken over the set $\mathfrak{M}(E)$ of all functions
$\varphi\in W_{0}(B_{16\rho}(x_{0}))$ with $\varphi\geqslant1$ on $E$. If $m=1$,
this definition leads to the standard definition of $Cp(\cdot)(E, B_{16\rho}(x_{0}))$
capacity (see, e.g. \cite{AlhutovKrash04}).


\subsection{Upper bound for the function $v$}
We note that in the standard case (i.e. if $L=0$) the upper bound for the function $v$
was proved in \cite{IVSkr1983}
(see also \cite[Chap.\,8, Sec.\,3]{IVSkrMetodsAn1994}, \cite{IVSkrSelWorks}).

\begin{lemma}\label{lem3.1}
There exists positive number $\gamma_{1}$ depending only on the data such that if
\begin{equation}\label{eq3.5}
\rho\,\frac{C_{p(\cdot)}(E, B_{16\rho}(x_{0});m)}{w(\frac{3}{2}\rho,16\rho)} \geqslant 1,
\end{equation}
then
\begin{equation}\label{eq3.6}
\gamma^{-1} \bigg(\rho^{p}\frac{C_{p(\cdot)}(E,B_{16\rho}(x_{0}),m)}{w\big(\frac{3}{2}\rho, 16\rho\big)}\bigg)^{\frac{1}{p-1}}\leqslant
M\big(3/2\rho\big)\leqslant  \gamma \bigg(\rho^{p}\frac{C_{p(\cdot)}(E,B_{16\rho}(x_{0}),m)}{w\big(\frac{3}{2}\rho, 16\rho\big)}\bigg)^{\frac{1}{p-1}},
\end{equation}
provided that
\begin{equation}\label{eq3.7}
\frac{1}{\log\log\dfrac{1}{16\rho}} + \gamma_{1}\,L\,\frac{\log\log\dfrac{1}{16\rho}}{\log\dfrac{1}{16\rho}} \leqslant 1.
\end{equation}
\end{lemma}
\begin{proof}
First we prove  inequality on the right-hand side of \eqref{eq3.6}. Fix $\sigma\in(0,1/4)$ and for any $s\in(5/4\rho, 2\rho(1-\sigma)), j=0,1,2,....$ set \\$\rho^{(1)}_{j}:=s(1+\sigma-\sigma 2^{-j}),\ \ K_{j}:=K(\rho^{(1)}_{j}, 16\rho), k_{j}:=k- k\,2^{-j},\ \ k>0,
A_{j}:=K_{j}\cap \{v> k_{j}\},\\  M_{j}:= \sup\limits_{K_{j}} v$ and let
$\zeta_{j}\in C^{\infty}(B_{16\rho}(x_{0})), \ \
0\leqslant\zeta_{j}\leqslant 1,  \zeta_{j}=0 \ \text{in} \ B_{\rho^{(1)}_{j}}(x_{0}),\ \
\zeta_{j}=1 \ \text{in} \ K_{j+1},\\
|\nabla \zeta_{j}|\leqslant \gamma\,2^{j}\big(\sigma\rho\big)^{-1}.$

Further we will assume that
$$ M\big(\frac{3}{2}\rho\big)\geqslant \frac{3}{2}\rho,$$
since otherwise, by \eqref{eq3.5}  inequality \eqref{eq3.6} is evident, moreover this inequality yields
$$\bigg(\frac{M(\rho_{1})}{\rho_{1}}\bigg)^{p(|x-x_{0}|)}\leqslant w(x,\rho_{1})\leqslant \gamma \bigg(\frac{M(\rho_{1})}{\rho_{1}}\bigg)^{p(|x-x_{0}|)},\quad x\in \mathcal{D}. $$
Testing \eqref{eq3.1} by $\varphi=(v-k_{j+1})_{+}\,\zeta_{j}^{\,q}$ and using the
Young inequality  we obtain
\begin{multline*}
\int\limits_{K_{j}\cap\{v>k_{j+1}\}}
|\nabla v|^{p(x)}\,\zeta_{j}^{\,q}\,dx
\leqslant \gamma\,\sigma^{-\gamma}\,2^{j\gamma}\rho^{-p}
\int\limits_{A_{j}}\bigg(\frac{M_{j}}{\rho^{(1)}_{j}}\bigg)^{p(|x-x_{0}|)}(v-k_{j})^{p}_{+}\,dx \leqslant \\
\leqslant \gamma\,\sigma^{-\gamma}\,2^{j\gamma}\rho^{-p}\int\limits_{A_{j}}\bigg(\frac{M_{0}}{\rho^{(1)}_{0}}\bigg)^{p(|x-x_{0}|)}(v-k_{j})^{p}_{+}\,dx=\gamma\,\sigma^{-\gamma}\,2^{j\gamma}\rho^{-p}\int\limits_{A_{j}}\,w(x,\rho^{(1)}_{0})(v-k_{j})^{p}_{+}\,dx.
\end{multline*}
Using again the Young inequality, assuming that $k > \varepsilon_{0}\,M_{0}$, where $\varepsilon_{0}\in (0, 1)$ is small enough, from the previous we have
\begin{multline*}
\int\limits_{K_{j}\cap\{v> k_{j+1}\}}w(x,\rho^{(1)}_{0})\,
|\nabla v|^{p}\,\zeta_{j}^{\,q}\,dx \leqslant \gamma\,\bigg(\frac{M_{0}}{\rho^{(1)}_{0}}\bigg)^{p}\int\limits_{K_{j}\cap\{v> k_{j+1}\}}w(x,\rho^{(1)}_{0})\,dx +\\ +\gamma\,\sigma^{-\gamma}\,2^{j\gamma}\rho^{-p}\int\limits_{A_{j}}\,w(x,\rho^{(1)}_{0})(v-k_{j})^{p}_{+}\,dx
\leqslant \gamma\,\sigma^{-\gamma}\,2^{j\gamma}\,\bigg(\frac{M_{0}}{\rho\,k}\bigg)^{p}\int\limits_{A_{j}}w(x,\rho^{(1)}_{0})(v-k_{j})^{p}_{+}\,dx +\\ +\gamma\,\sigma^{-\gamma}\,2^{j\gamma}\rho^{-p}\int\limits_{A_{j}}\,w(x,\rho^{(1)}_{0})(v-k_{j})^{p}_{+}\,dx \leqslant
\gamma\,\varepsilon^{-\gamma}_{0}\,\sigma^{-\gamma}\,2^{j\gamma}\rho^{-p}\int\limits_{A_{j}}\,w(x,\rho^{(1)}_{0})(v-k_{j})^{p}_{+}\,dx.
\end{multline*}
Choose $\gamma_{1}>0$ large enough, by our assumption Lemma \ref{lem2.2} is applicable, therefore from this we obtain
\begin{multline*}
y_{j+1}:=\int\limits_{A_{j+1}}\,w(x,\rho^{(1)}_{0})(v-k_{j+1})^{p}_{+}\,dx \leqslant\\ \leqslant \gamma\,\varepsilon^{-\gamma}_{0}\,\sigma^{-\gamma}\,2^{j\gamma}\,[w(\rho^{(1)}_{0},16\rho)]^{-(1-\frac{1}{\kappa_{1}})}\,\bigg(\int\limits_{A_{j+1}}w(x,\rho^{(1)}_{0})\,dx\bigg)^{1-\frac{1}{\kappa_{1}}}\int\limits_{A_{j}}\,w(x,\rho^{(1)}_{0})(v-k_{j})^{p}_{+}\,dx \leqslant \\ \leqslant \gamma\,\varepsilon^{-\gamma}_{0}\,\sigma^{-\gamma}\,2^{j\gamma}\,[w(\rho^{(1)}_{0},16\rho)]^{-(1-\frac{1}{\kappa_{1}})}\,k^{-p(1-\frac{1}{\kappa_{1}})}\,y^{2-\frac{1}{\kappa_{1}}}_{j},
j=0,1,2,...
\end{multline*}
Hence, setting $M_{\sigma}:=M_{\infty}$, by standard arguments (see, e.g. \cite{LadUr}) and by our choice, we arrive at
\begin{equation}\label{eq3.8}
M^{p}_{\sigma} \leqslant \varepsilon^{p}_{0}\,M^{p}_{0}+\gamma\,\varepsilon^{-\gamma}_{0}\,\sigma^{-\gamma}\,[w(\rho^{(1)}_{0},16\rho)]^{-1}\,\int\limits_{K(\rho^{(1)}_{0},16\rho)}w(x,\rho^{(1)}_{0})\,
v^{p}\,dx .
\end{equation}

Let us estimate the  second term on the right-hand side of \eqref{eq3.8}.
For this we set \\ $v_{M_{0}}:=\min\{v, M_{0}\}$,
by Lemma \ref{lem2.2} we have for any $\varepsilon\in (0,1)$
\begin{multline*}
\int\limits_{K(\rho^{(1)}_{0},16\rho)}w(x,\rho^{(1)}_{0})v^{p}dx=\int\limits_{K(\rho^{(1)}_{0},16\rho)}w(x,\rho^{(1)}_{0})
v^{p}_{M_{0}}dx\leqslant \gamma\rho^{p}\int\limits_{K(\rho^{(1)}_{0},16\rho)}w(x,\rho^{(1)}_{0})|\nabla v_{M_{0}}|^{p}dx\leqslant
\\ \leqslant \varepsilon\,\gamma^{-1}\,\sigma^{\gamma}\,\rho^{p}\, \int\limits_{K(\rho^{(1)}_{0},16\rho)}[w(x,\rho^{(1)}_{0})]^{\frac{p(x)}{p(x)-p}}\,dx + \gamma\,\varepsilon^{-\gamma}\,\sigma^{-\gamma}\,\rho^{p}\,\int\limits_{\mathcal{D}}|\nabla v_{M_{0}}|^{p(x)}\,dx= \\
= \varepsilon\,\gamma^{-1}\,\sigma^{\gamma}\,M^{p}_{0}\,w(\rho^{(1)}_{0},16\rho)+\gamma\,\varepsilon^{-\gamma}\,\sigma^{-\gamma}\,\rho^{p}\,\int\limits_{\mathcal{D}}|\nabla v_{M_{0}}|^{p(x)}\,dx.
\end{multline*}
Collecting the last two inequalities we obtain
\begin{equation}\label{eq3.9}
M^{p}_{\sigma}\leqslant (\varepsilon^{p}_{0}+ \frac{\varepsilon}{\varepsilon^{\gamma}_{0}})\,M^{p}_{0}+ \gamma\,\varepsilon^{-\gamma}_{0}\,\varepsilon^{-\gamma}\,\sigma^{-\gamma}\,\rho^{p}\,[w(\rho^{(1)}_{0},16\rho)]^{-1}\,\int\limits_{\mathcal{D}}|\nabla v_{M_{0}}|^{p(x)}\,dx .
\end{equation}

Let us estimate the second term on the right-hand side of \eqref{eq3.9}. Let $\psi\in \mathfrak{M}(E)$ be such that
$$
\frac{1}{m}\int\limits_{B_{16\rho}(x_{0})}
|m\,\nabla \psi|^{p(x)}\,dx
\leqslant
C_{p(\cdot)}(E,B_{16\rho}(x_{0});m)+ \rho^{n}\leqslant C_{p(\cdot)}(E,B_{16\rho}(x_{0});m)+ \gamma w(\rho^{(1)}_{0},16\rho) .
$$
Testing identity \eqref{eq3.1} by $\varphi=v-m\psi$,
by the Young inequality  we obtain
\begin{equation*}
\int\limits_{\mathcal{D}} |\nabla v|^{p(x)}\,dx
\leqslant \gamma  \int\limits_{B_{16\rho}(x_{0})}
|m\nabla \psi|^{p(x)}\,dx\leqslant
\gamma m\left( C_{p(\cdot)}(E, B_{16\rho}(x_{0});m)+ w(\rho^{(1)}_{0},16\rho) \right).
\end{equation*}
Testing \eqref{eq3.1} by $\varphi=v_{M_{0}}-\dfrac{M_{0}}{m}\,v$,
using the Young inequality and the previous inequality, we have
$$
\int\limits_{\mathcal{D}} |\nabla v_{M_{0}}|^{p(x)}\,dx
\leqslant \gamma \frac{ M_{0}}{m}
\int\limits_{\mathcal{D}} |\nabla v|^{p(x)}\,dx
\leqslant
\gamma M_{0}\big( C_{p(\cdot)}(E, B_{16\rho}(x_{0});m)+ w(\rho^{(1)}_{0},16\rho) \big).
$$
This inequality, the Young inequality  and \eqref{eq3.9} imply that
\begin{multline*}
M^{p}_{\sigma} \leqslant (\varepsilon^{p}_{0}+\frac{\varepsilon}{\varepsilon^{\gamma}_{0}})\,M^{p}_{0} + \gamma\,\varepsilon^{-\gamma}_{0}\,\varepsilon^{-\gamma}\,\sigma^{-\gamma}M_{0}\bigg( \rho^{p}\,\frac{C_{p(\cdot)}(E, B_{16\rho}(x_{0});m)}{w(\rho^{(1)}_{0},16\rho)}+ \rho^{p} \bigg)\leqslant \\ \leqslant
(2\varepsilon^{p}_{0}+\frac{\varepsilon}{\varepsilon^{\gamma}_{0}})\,M^{p}_{0} +\gamma\,\varepsilon^{-\gamma}_{0}\,\varepsilon^{-\gamma}\,\sigma^{-\gamma}\bigg\{\bigg( \rho^{p}\,\frac{C_{p(\cdot)}(E, B_{16\rho}(x_{0});m)}{w(\frac{3}{2}\rho,16\rho)}\bigg)^{\frac{p}{p-1}}+ \rho^{p}\bigg\} ,
\end{multline*}
Iterating the last inequality, choosing $\varepsilon_{0}$ and then $\varepsilon=\varepsilon(\varepsilon_{0})$ small enough, by \eqref{eq3.5}
we arrive at
$$
M(3/2\rho) \leqslant \gamma\,\bigg( \rho^{p}\,\frac{C_{p(\cdot)}(E, B_{16\rho}(x_{0});m)}{w(\frac{3}{2}\rho,16\rho)}\bigg)^{\frac{1}{p-1}}+ \gamma\,\rho \leqslant \gamma\,\bigg( \rho^{p}\,\frac{C_{p(\cdot)}(E, B_{16\rho}(x_{0});m)}{w(\frac{3}{2}\rho,16\rho)}\bigg)^{\frac{1}{p-1}},
$$
which completes the proof of the lemma.

Now we prove inequality on the left-hand side of \eqref{eq3.6}. Let $\zeta_{1}\in C_{0}^{\infty}(B_{4\rho}(x_{0}))$, $0\leqslant \zeta_{1}\leqslant1$, $\zeta_{1}=1$ in $B_{2\rho}(x_{0})$, $|\nabla \zeta_{1}|\leqslant \dfrac{\gamma}{\rho}$. Testing
\eqref{eq3.1} by $\varphi=v-m\,\zeta_{1}^{\,q}$ , using  the Young inequality ,
we obtain for any $\varepsilon_{1} >0$
$$
\begin{aligned}
\int\limits_{\mathcal{D}} |\nabla v|^{p(x)}\,dx
&\leqslant
\gamma\,\frac{ m}{\rho} \int\limits_{K(2\rho,4\rho)}
|\nabla v|^{p(x)-1}\,\zeta_{1}^{\,q-1} \,dx
\\
&\leqslant \gamma\,\frac{m}{\varepsilon_{1}\rho}\int\limits_{K(2\rho,4\rho)}
 |\nabla v|^{p(x)}\,dx+\gamma\,\frac{ m}{\rho}\int\limits_{K(2\rho,4\rho)}
\varepsilon_{1}^{p(x)-1}\,dx.
\end{aligned}
$$
Let $\zeta_{2}\in C_{0}^{\infty}(K(\frac{3}{2}\rho,\, 6\rho))$, $0\leqslant \zeta_{2}\leqslant1$,
$\zeta_{2}=1$ in $K(2\rho,\, 4\rho)$, $|\nabla \zeta_{2}|\leqslant \dfrac{\gamma}{\rho}$.
Testing \eqref{eq3.1} by $\varphi=v\,\zeta_{2}^{\,q}$ and using the Young inequality,
we estimate the first term on the right-hand side of the previous inequality as follows:
$$
\int\limits_{K(2\rho,4\rho)} |\nabla v|^{p(x)}\,dx \leqslant \gamma \int\limits_{K(\frac{3}{2}\rho,6\rho)}
\bigg(\frac{v}{\rho}\bigg)^{p(x)}\,dx \leqslant \frac{\gamma}{\rho^{p}}\int\limits_{K(\frac{3}{2}\rho,16\rho)}\,\,w(x,\frac{3}{2}\rho)\,v^{p}\,dx.
$$
Combining the last two inequalities and using the definition of capacity,
we obtain
\begin{multline}\label{eq3.10}
C_{p(\cdot)}(E, B_{16\rho}(x_{0});m)\leqslant
\frac{1}{m}\int\limits_{\mathcal{D}} |\nabla v|^{p(x)}\,dx
\leqslant \\ \leqslant
\frac{\gamma}{\varepsilon_{1}\rho^{p+1}}\int\limits_{K(\frac{3}{2}\rho,16\rho)}\,\,w(x,\frac{3}{2}\rho)\,v^{p}\,dx+
\frac{\gamma}{\rho}\int\limits_{K(2\rho,4\rho)}\varepsilon_{1}^{p(x)-1}\,dx.
\end{multline}
Choose $\varepsilon_{1}$ from the condition  $\varepsilon_{1}=\dfrac{M(\frac{3}{2}\rho)}{\rho}$, then inequality \eqref{eq3.10} yields
$$
C_{p(\cdot)}(E, B_{16\rho}(x_{0});m)\leqslant \gamma w(\frac{3}{2}\rho,16\rho)\,\frac{M^{p-1}(\frac{3}{2}\rho)}{\rho^{p}},
$$
from which the required inequality follows, this completes the proof of the lemma.
\end{proof}

\subsection{Lower bound for the function $v$}

Further we need   the following lemma.
\begin{lemma}\label{lem3.2}
Let condition \eqref{eq3.5} holds, then there exists  $\varepsilon \in(0,1)$ depending only on the data such that
\begin{equation}\label{eq3.11}
\int\limits_{K(\frac{3}{2}\rho, 16\rho)} w(x,\frac{3}{2}\rho)\chi\bigg[v\geqslant \varepsilon \bigg( \rho^{p}\,\frac{C_{p(\cdot)}(E, B_{16\rho}(x_{0});m)}{w(\frac{3}{2}\rho,16\rho)}\bigg)^{\frac{1}{p-1}}\bigg]\,dx \geqslant \gamma^{-1} w(\frac{3}{2}\rho, 16\rho),
\end{equation}
provided that inequality \eqref{eq3.7} holds.
\end{lemma}
\begin{proof}
To prove  \eqref{eq3.11} we use inequality \eqref{eq3.10}.
Choose $\varepsilon_{1}$ from the condition  $\varepsilon_{1}=\bar{\varepsilon}_{1} \dfrac{M(\frac{3}{2}\rho)}{\rho} ,\\ \bar{\varepsilon}_{1}\in(0,1)$, then by Lemma \ref{lem3.1} the terms on the right-hand side of \eqref{eq3.10} are estimated as follows
\begin{equation}\label{eq3.12}
\frac{\gamma}{\rho}\int\limits_{K_{2\rho,4\rho}}\varepsilon_{1}^{p(x)-1}\,dx \leqslant
\gamma \frac{ \bar{\varepsilon}_{1}}{\rho^{p}}M(3/2\rho)^{p-1}w(\frac{3}{2}\rho, 16\rho) \leqslant \gamma\bar{\varepsilon}_{1}C_{p(\cdot)}(E, B_{16\rho}(x_{0});m).
\end{equation}
Similarly, by Lemma \ref{lem3.1}
\begin{multline}\label{eq3.13}
\frac{\gamma}{\varepsilon_{1}\rho^{p}}\int\limits_{K_{\frac{3}{2}\rho,16\rho}}\,\,w(x,\frac{3}{2}\rho)\,v^{p}\,dx \leqslant
\gamma\,\frac{\varepsilon^{p}}{\bar{\varepsilon}_{1}}\,C_{p(\cdot)}(E, B_{16\rho}(x_{0});m)+\\ +
\gamma\,\frac{C_{p(\cdot)}(E, B_{16\rho}(x_{0});m)}{w(\frac{3}{2}\rho,16\rho)}\int\limits_{K(\frac{3}{2}\rho, 16\rho)} w(x,\frac{3}{2}\rho)\chi\bigg[v\geqslant \varepsilon \bigg( \rho^{p}\,\frac{C_{p(\cdot)}(E, B_{16\rho}(x_{0});m)}{w(\frac{3}{2}\rho,16\rho)}\bigg)^{\frac{1}{p-1}}\bigg]\,dx.
\end{multline}
Collecting estimates \eqref{eq3.10}, \eqref{eq3.12}, \eqref{eq3.13} we obtain
\begin{multline*}
C_{p(\cdot)}(E, B_{16\rho}(x_{0});m) \leqslant \big(\gamma \bar{\varepsilon}_{1} + \gamma\,\frac{\varepsilon^{p}}{\bar{\varepsilon}_{1}}\big)\,C_{p(\cdot)}(E, B_{16\rho}(x_{0});m)+ \\ +
\gamma\,\frac{C_{p(\cdot)}(E, B_{16\rho}(x_{0});m)}{w(\frac{3}{2}\rho,16\rho)}\int\limits_{K(\frac{3}{2}\rho, 16\rho)} w(x,\frac{3}{2}\rho)\chi\bigg[v\geqslant \varepsilon \bigg( \rho^{p}\,\frac{C_{p(\cdot)}(E, B_{16\rho}(x_{0});m)}{w(\frac{3}{2}\rho,16\rho)}\bigg)^{\frac{1}{p-1}}\bigg]\,dx.
\end{multline*}
Choosing $\bar{\varepsilon}_{1}$ by the condition $\gamma \bar{\varepsilon}_{1} =\dfrac{1}{4}$ and then choosing $\varepsilon$ by the condition $\gamma\,\dfrac{\varepsilon^{p}}{\bar{\varepsilon}_{1}}=\dfrac{1}{4}$, from the previous we arrive at
\begin{equation*}
\int\limits_{K(\frac{3}{2}\rho, 16\rho)} w(x,\frac{3}{2}\rho)\chi\bigg[v\geqslant \varepsilon \bigg( \rho^{p}\,\frac{C_{p(\cdot)}(E, B_{16\rho}(x_{0});m)}{w(\frac{3}{2}\rho,16\rho)}\bigg)^{\frac{1}{p-1}}\bigg]\,dx \geqslant \gamma^{-1} w(\frac{3}{2}\rho, 16\rho),
\end{equation*}
which completes the proof of the lemma.
\end{proof}
The following lemma is the main result of this Section
\begin{lemma}\label{lem3.3}
There exists $\bar{\varepsilon} \in(0,1)$ depending only on the data  such that either
\begin{equation}\label{eq3.14}
\bar{\varepsilon}\,m\,\bigg(\frac{|E|}{|B_{\rho}(x_{0})|}\bigg)^{\frac{1}{p-1}}\leqslant \rho,
\end{equation}
or
\begin{equation}\label{eq3.15}
\big|\big\{K(\frac{3}{2}\rho,16\rho) : v \geqslant \bar{\varepsilon}\,m\,\bigg(\frac{|E|}{|B_{\rho}(x_{0})|}\bigg)^{\frac{1}{p-1}}\big\}\big| \geqslant \gamma^{-1}\,|K(\frac{3}{2}\rho,16\rho)|,
\end{equation}
provided that inequality \eqref{eq3.7} holds.
\end{lemma}
\begin{proof}
Lemma \ref{lem3.3} is a consequence of Lemma \ref{lem3.2}, for this we first estimate the capacity of the set $E$ from below. Let $\varphi \in W_{0}(B_{16\rho}(x_{0})), \varphi =1 $ on $E$ , then by Lemmas \ref{lem2.1}, \ref{lem2.2}, \ref{lem3.1} and using the evident inequalities
$\gamma^{-1} w(\frac{3}{2}\rho, 16\rho)\leqslant \int\limits_{B_{16\rho}(x_{0})}w(x,\frac{3}{2}\rho)\,dx \leqslant\gamma w(\frac{3}{2}\rho, 16\rho)$  we have
\begin{multline}\label{eq3.16}
m^{p}\,|E| \leqslant m^{p}\int\limits_{B_{16\rho}(x_{0})}\varphi^{p}\,dx\leqslant \\ \leqslant
\bigg(\int\limits_{B_{16\rho}(x_{0})}w(x,\frac{3}{2}\rho)|(m\varphi)|^{p\kappa_{1}}\,dx\bigg)^{\frac{1}{\kappa_{1}}} \bigg(\int\limits_{B_{16\rho}(x_{0})}[w(x,\frac{3}{2}\rho)]^{-\frac{1}{\kappa_{1}-1}}\,dx\bigg)^{1-\frac{1}{\kappa_{1}}}\leqslant \\ \leqslant
\gamma\,\rho^{p}\,[w(\frac{3}{2}\rho,16\rho)]^{\frac{1}{\kappa_{1}}-1}
\bigg(\int\limits_{B_{16\rho}(x_{0})}[w(x,\frac{3}{2}\rho)]^{-\frac{1}{\kappa_{1}-1}}\,dx\bigg)^{1-\frac{1}{\kappa_{1}}}
\int\limits_{B_{16\rho}(x_{0})}w(x,\frac{3}{2}\rho)|\nabla(m\varphi)|^{p}\,dx\leqslant \\ \leqslant
\gamma\frac{\rho^{p+n}}{w(\frac{3}{2}\rho,16\rho)}\,\int\limits_{B_{16\rho}(x_{0})}w(x,\frac{3}{2}\rho)|\nabla(m\varphi)|^{p}\,dx
\leqslant\\ \leqslant \gamma\frac{\rho^{p+n}}{w(\frac{3}{2}\rho,16\rho)}\bigg(\int\limits_{B_{16\rho}(x_{0})}w(x,\frac{3}{2}\rho)\bigg(1+\frac{M(\frac{3}{2}\rho)}{\rho}\bigg)^{p}\,dx + \int\limits_{B_{16\rho}(x_{0})}|\nabla(m\varphi)|^{p(x)}\,dx\bigg) \leqslant\\ \leqslant
\gamma \rho^{p+n} +\gamma\,m\, \frac{\rho^{p+n}}{w(\frac{3}{2}\rho,16\rho)}\,C_{p(\cdot)}(E, B_{16\rho}(x_{0}), m)
+ \gamma\frac{\rho^{p+n}}{w(\frac{3}{2}\rho,16\rho)}\int\limits_{B_{16\rho}(x_{0})}|\nabla(m\varphi)|^{p(x)}\,dx.
\end{multline}
Since $\varphi$ is arbitrary, estimate \eqref{eq3.16} yields
\begin{equation*}
m^{p}\,|E| \leqslant \gamma \rho^{p+n} + \gamma\,m\, \frac{\rho^{p+n}}{w(\frac{3}{2}\rho,16\rho)}\,C_{p(\cdot)}(E, B_{16\rho}(x_{0}), m).
\end{equation*}
If inequality \eqref{eq3.14} is violated then
\begin{equation*}
m^{p}\frac{|E|}{|B_{\rho}(x_{0})|}\geqslant m^{p}\bigg(\frac{|E|}{|B_{\rho}(x_{0})|}\bigg)^{\frac{p}{p-1}} \geqslant\bigg(\frac{\rho}{\bar{\varepsilon}}\bigg)^{p} ,
\end{equation*}
so, if $\bar{\varepsilon}$ is sufficiently small, from the previous we arrive at
\begin{equation*}
\frac{\rho^{p}}{w(\frac{3}{2}\rho,16\rho)}\,C_{p(\cdot)}(E, B_{16\rho}(x_{0}), m) \geqslant \gamma^{-1}\,m^{p-1}\,\frac{|E|}{|B_{\rho}(x_{0})|}.
\end{equation*}
And hence
\begin{equation*}
\rho\,\frac{C_{p(\cdot)}(E, B_{16\rho}(x_{0});m)}{w(\frac{3}{2}\rho,16\rho)} \geqslant \gamma^{-p}\,\rho^{1-p}m^{p-1}\,\frac{|E|}{|B_{\rho}(x_{0})|}\geqslant \gamma^{-p}\,\bar{\varepsilon}^{-p} \geqslant \gamma_{0},
\end{equation*}
provided that \eqref{eq3.14} is violated and $\bar{\varepsilon}$ is sufficiently small. Now we use Lemma \ref{lem3.2} for this we set
$F:=\big\{K(\frac{3}{2}\rho,16\rho) : v \geqslant \bar{\varepsilon}\,m\,\bigg(\dfrac{|E|}{|B_{\rho}(x_{0})|}\bigg)^{\frac{1}{p-1}}\big\}$  and
$w(F):= \int\limits_{K(\frac{3}{2}\rho, 16\rho)} w(x,\frac{3}{2}\rho)\chi(F)\,dx.$\\ We have by Lemmas \ref{lem2.1} and \ref{lem3.2}
\begin{equation*}
\gamma^{-1}\leqslant \frac{w(F)}{w(\frac{3}{2}\rho,16\rho)} \leqslant \frac{|F|^{\frac{1}{2}}}{w(\frac{3}{2}\rho,16\rho)}
\bigg(\int\limits_{K(\frac{3}{2}\rho,16\rho)}[w(x,\frac{3}{2}\rho)]^{2}\,dx\bigg)^{\frac{1}{2}}\leqslant
\gamma \bigg(\frac{|F|}{|K(\frac{3}{2}\rho,16\rho)|}\bigg)^{\frac{1}{2}},
\end{equation*}
which completes the proof of the lemma.
\end{proof}


\section{Harnack's inequality,
proof of Theorems \ref{th1.2} and \ref{th1.3}}\label{Sec4}

\subsection{Weak Harnack inequality, proof of Theorem \ref{th1.2}}
For $0<m<M$ set $E(\rho,m):=\big\{B_{\rho}(x_{0}): u\geqslant m \big\}.$ As it was mentioned in Section $1$ Theorem \ref{th1.2} is a simple consequence of the following lemma
\begin{lemma}\label{lem4.1}
Let $u$ be a non-negative bounded super-solution to equation \eqref{eq1.1} in $\Omega$ and let condition \eqref{eq1.8} be fulfilled, then there exist positive numbers $C_{2}, C_{3}$ depending only on the data such that
\begin{equation}\label{eq4.1}
|E(\rho,m)|\leqslant C_{2}\,|B_{\rho}(x_{0})|\,m^{1-p}\,\big(\rho + \min\limits_{B_{\frac{\rho}{2}}(x_{0})} u\big)^{p-1},
\end{equation}
provided that $B_{16\rho}(x_{0})\subset \Omega$ and
\begin{equation}\label{eq4.2}
\frac{1}{\log\log\dfrac{1}{16\rho}} + C_{3}\,L\,\frac{\log\log\dfrac{1}{16\rho}}{\log\dfrac{1}{16\rho}} \leqslant 1.
\end{equation}
\end{lemma}
\begin{proof}
We construct the solution $v$ of the problem \eqref{eq3.1} in $\mathcal{D}=B_{16\rho}(x_{0})\setminus E(\rho,m)$, since $u\geqslant v$ on $\partial \mathcal{D}$, by \eqref{eq1.8} $u\geqslant v$ in $\mathcal{D}$. First we use Lemma \ref{lem3.3}, if inequality \eqref{eq3.14}
is violated, i.e. if
\begin{equation}\label{eq4.3}
\bar{\varepsilon}\,m\,\bigg(\frac{|E(\rho,m)|}{|B_{\rho}(x_{0})|}\bigg)^{\frac{1}{p-1}}\geqslant \rho,
\end{equation}
by Lemma \ref{lem3.3} there holds
\begin{multline*}
\big|\big\{B_{16\rho}(x_{0}) : u \geqslant \bar{\varepsilon}\,m\,\bigg(\frac{|E(\rho,m)|}{|B_{\rho}(x_{0})|}\bigg)^{\frac{1}{p-1}}\big\}\big| \geqslant \\ \geqslant \big|\big\{K(\frac{3}{2}\rho,16\rho) : u \geqslant \bar{\varepsilon}\,m\,\bigg(\frac{|E(\rho,m)|}{|B_{\rho}(x_{0})|}\bigg)^{\frac{1}{p-1}}\big\}\big| \geqslant \\ \geqslant \big|\big\{K(\frac{3}{2}\rho,16\rho) : v \geqslant \bar{\varepsilon}\,m\,\bigg(\frac{|E(\rho,m)|}{|B_{\rho}(x_{0})|}\bigg)^{\frac{1}{p-1}}\big\}\big|  \geqslant \gamma^{-1}\,|B_{16\rho}(x_{0})|,
\end{multline*}
provided that
\begin{equation}\label{eq4.4}
\frac{1}{\log\log\dfrac{1}{16\rho}} + \gamma_{1} \,L\,\frac{\log\log\dfrac{1}{16\rho}}{\log\dfrac{1}{16\rho}} \leqslant 1.
\end{equation}
Now we use Lemma \ref{lem2.5} with $r=16 \rho$, $\mu_{-}=0$~~ and~~ $\xi\,\omega_{r}=\bar{\varepsilon}\,m\,\bigg(\dfrac{|E(\rho,m)|}{|B_{\rho}(x_{0})|}\bigg)^{\frac{1}{p-1}}$, we obtain that
\begin{equation}\label{eq4.5}
u(x)\,\geqslant 2^{-s_{*}-1}\,m\,\bigg(\dfrac{|E(\rho,m)|}{|B_{\rho}(x_{0})|}\bigg)^{\frac{1}{p-1}},\quad x\in B_{8\rho}(x_{0}),
\end{equation}
provided that
\begin{equation}\label{eq4.6}
\frac{1}{\log\log\dfrac{1}{16\rho}} + s_{*}\,L\,\frac{\log\log\dfrac{1}{16\rho}}{\log\dfrac{1}{16\rho}} \leqslant 1.
\end{equation}
Choosing $C_{2}, C_{3}$ sufficiently large, collecting  \eqref{eq4.3}--\eqref{eq4.6} we arrive at \eqref{eq4.1}, which completes the proof of the lemma.
\end{proof}
To complete the proof of Theorem \ref{th1.2}  set $m(\rho/2):= \rho/2 + \min\limits_{B_{\frac{\rho}{2}}(x_{0})} u $, then by Lemma \ref{lem4.1} for $\theta\in(0, p-1)$ we have
\begin{multline*}
|B_{\rho}(x_{0})|^{-1}\,\int\limits_{B_{\rho}(x_{0})}u^{\theta}\,dx= \theta\,|B_{\rho}(x_{0})|^{-1}\,\int\limits_{0}^{\infty}|E(\rho,m)|\,m^{\theta-1}\,dm\leqslant m^{\theta}(\rho/2) +\\ +\theta\,|B_{\rho}(x_{0})|^{-1}\,\int\limits_{m(\rho/2)}^{\infty}|E(\rho,m)|\,m^{\theta-1}\,dm
\leqslant m^{\theta}(\rho/2) +\gamma\,m^{p-1}(\rho/2)\,\int\limits_{m(\rho/2)}^{\infty} m^{\theta -p}\,dm \leqslant \\
\leqslant \frac{\gamma}{p-1-\theta}\,m^{\theta}(\rho/2),
\end{multline*}
provided that
\begin{equation*}
\frac{1}{\log\log\dfrac{1}{16\rho}} + C_{3}\,L\,\frac{\log\log\dfrac{1}{16\rho}}{\log\dfrac{1}{16\rho}} \leqslant 1,
\end{equation*}
which completes the proof of Theorem \ref{th1.2}.


\subsection{Proof of Theorem \ref{th1.3}}
The proof of Theorem \ref{th1.3} is almost standard. For fixed $\sigma \in(0,1/8), s\in (3/4\rho, 7/8\rho), k>0$ and
$j= 0, 1, 2,...$ set $k_{j}:= k- k 2^{-j}, \rho_{j}:= s(1-\sigma +\sigma 2^{-j}), \bar{\rho}_{j}:=\dfrac{1}{2}(\rho_{j}+\rho_{j+1}),
B_{j}:=B_{\rho_{j}}(x_{0}),\\ \bar{B}_{j}:=B_{\bar{\rho}_{j}}(x_{0})$ and let $M_{0}:=\max\limits_{B_{0}} u,
M_{\sigma}:=\max\limits_{B_{\infty}} u$. Denote by $\zeta_{j}$ a non-negative piecewise smooth cutoff function in $\bar{B}_{j}$
that equals one on $B_{j+1}$, such that $|\nabla \zeta_{j}|\leqslant \gamma\dfrac{2^{j}}{\sigma\rho}.$ Set also \\
$w_{0}(x):=\bigg(1+\dfrac{M_{0}}{\rho_{0}}\bigg)^{p(|x-x_{0}|)}$ and $w_{0}(F):= \int\limits_{F}\,w_{0}(x)\,dx$. Evidently we have\\
$\bigg(\dfrac{M_{0}}{\rho_{0}}\bigg)^{p(|x-x_{0}|)} \leqslant w_{0}(x) \leqslant \gamma\bigg(\dfrac{M_{0}}{\rho_{0}}\bigg)^{p(|x-x_{0}|)} $,
if $M_{0} \geqslant \rho_{0}$.

Note that similarly to \eqref{eq2.5}, \eqref{eq2.6}  there hold
\begin{equation}\label{eq4.7}
\gamma^{-1}\bigg(1+\frac{M_{0}}{\rho_{0}}\bigg)^{-tp(\rho_{0})}\leqslant |B_{0}|^{-1}\,\int\limits_{B_{0}} w^{-t}_{0}(x)\,dx\leqslant \gamma \bigg(1+\frac{M_{0}}{\rho_{0}}\bigg)^{-tp(\rho_{0})},  t>0,
\end{equation}
\begin{equation}\label{eq4.8}
\gamma^{-1}\bigg(1+\frac{M_{0}}{\rho_{0}}\bigg)^{tp(\rho_{0})}\leqslant |B_{0}|^{-1}\,\int\limits_{B_{0}} w^{t}_{0}(x)\,dx\leqslant \gamma \bigg(1+\frac{M_{0}}{\rho_{0}}\bigg)^{tp(\rho_{0})},\quad  t>0,
\end{equation}
provided that
\begin{equation}\label{eq4.9}
\frac{1}{\log\log\dfrac{1}{\rho_{0}}} + t\,\bar{C}\,L\,\log(1+\frac{M}{\rho_{0}})\frac{\log\log\dfrac{1}{\rho_{0}}}{\log^{2}\dfrac{1}{\rho_{0}}} \leqslant 1,
\end{equation}
where $\bar{C}=\max(C, C_{1})$ and $C$, $C_{1}$ are the constants defined in  Lemmas  \ref{lem2.1}, \ref{lem2.2}. Therefore Lemmas \ref{lem2.1}, \ref{lem2.2} continue to hold in $B_{0}$ with $w_{\pm}(x, u, k, r)$ replaced by $w_{0}(x)$.

Furthter we will assume that $M_{0} \geqslant \rho_{0}.$ Test identity \eqref{eq1.7} by $\varphi=(u-k_{j+1})_{+}\,\zeta^{q}_{j}$, then
\begin{multline*}
\int\limits_{\bar{B}_{j}}\,| \nabla(u-k_{j+1})_{+}|^{p(x)}\,\zeta^{q}_{j}\,dx \leqslant
\gamma\,2^{j\gamma}\,\int\limits_{\bar{B}_{j}}\bigg(\frac{u-k_{j+1}}{\sigma\rho}\bigg)_{+}^{p(x)}\,dx \leqslant \\
\leqslant \gamma\,\sigma^{-\gamma}\,\frac{2^{j\gamma}}{\rho^{p}}\,\int\limits_{B_{j}}\,w_{0}(x)\,(u-k_{j})_{+}^{p}\,dx.
\end{multline*}
From this by the Young inequality, assuming that $k> \varepsilon_{0}\,M_{0}$, $\varepsilon_{0}\in (0, 1)$ is small enough, we obtain
\begin{multline*}
\int\limits_{\bar{B}_{j}}\,w_{0}(x)\,| \nabla(u-k_{j+1})_{+}|^{p}\,\zeta^{q}_{j}\,dx \leqslant
\gamma \bigg(\frac{M_{0}}{\rho}\bigg)^{p}\,\int\limits_{\bar{B}_{j}\cap\{u> k_{j+1}\}}\,w_{0}(x)\,dx+\\
 \gamma\,\sigma^{-\gamma}\,\frac{2^{j\gamma}}{\rho^{p}}\,\,\int\limits_{B_{j}}w_{0}(x)\,(u-k_{j+1})_{+}^{p}\,dx \leqslant
\gamma\,\sigma^{-\gamma}\,2^{\j\gamma}\bigg(\frac{M_{0}}{\rho\,k}\bigg)^{p}\,\int\limits_{B_{j}\cap\{u> k_{j}\}}\,w_{0}(x)\,(u-k_{j})^{p}_{+}\,dx+\\ +\gamma\,\sigma^{-\gamma}\,\frac{2^{j\gamma}}{\rho^{p}}\,\,\int\limits_{B_{j}}w_{0}(x)\,(u-k_{j+1})_{+}^{p}\,dx
\leqslant \gamma\,\varepsilon^{-\gamma}_{0}\,\sigma^{-\gamma}\,\frac{2^{j\gamma}}{\rho^{p}}\,\,\int\limits_{B_{j}}w_{0}(x)\,(u-k_{j+1})_{+}^{p}\,dx,
\end{multline*}
provided that \eqref{eq4.9} holds.  From this similarly to \eqref{eq3.8} we obtain
\begin{equation}\label{eq4.10}
M_{\sigma}\leqslant \varepsilon^{\frac{1}{p}}_{0}\,M_{0}+\gamma\,\varepsilon^{-\gamma}_{0}\,\sigma^{-\gamma}\,
\bigg([w_{0}(B_{0})]^{-1}\int\limits_{B_{0}}\,w_{0}(x)\,u^{p}\,dx\bigg)^{\frac{1}{p}}.
\end{equation}
Let us estimate the second term on the right-hand side of \eqref{eq4.10}, using  Lemma \ref{lem2.1} we obtain for any $0< \theta <p$
and any $\varepsilon \in(0, 1)$
\begin{multline*}
\bigg([w_{0}(B_{0})]^{-1}\int\limits_{B_{0}}\,w_{0}(x)\,u^{p}\,dx\bigg)^{\frac{1}{p}}
\leqslant M^{1-\frac{\theta}{2p}}_{0}\,\bigg([w_{0}(B_{0})]^{-1}\int\limits_{B_{0}}\,w_{0}(x)\,u^{\frac{\theta}{2}}\,dx\bigg)^{\frac{1}{p}}\leqslant \\ \leqslant \varepsilon M_{0} +\gamma \varepsilon^{-\gamma}\bigg([w_{0}(B_{0})]^{-1}\int\limits_{B_{0}}\,w_{0}(x)\,u^{\frac{\theta}{2}}\,dx\bigg)^{\frac{2}{\theta}} \leqslant\\ \leqslant \varepsilon M_{0}  +\gamma \varepsilon^{-\gamma}[w_{0}(B_{0})]^{-\frac{2}{\theta}}
\bigg(\int\limits_{B_{0}}\,w^{2}_{0}(x)\,dx\bigg)^{\frac{1}{\theta}}\,\bigg(\int\limits_{B_{0}}\,u^{\theta}\,dx\bigg)^{\frac{1}{\theta}}\leqslant \\ \leqslant \varepsilon M_{0} + \gamma \varepsilon^{-\gamma}\,\bigg(\rho^{-n}\,\int\limits_{B_{0}}\,u^{\theta}\,dx\bigg)^{\frac{1}{\theta}},
\end{multline*}
which together with \eqref{eq4.10} yield
\begin{equation*}
M_{\sigma}\leqslant (\varepsilon^{\frac{1}{p}}_{0}+\varepsilon)\,M_{0}+\gamma\,\varepsilon^{-\gamma}_{0}\,\varepsilon^{-\gamma}\,\sigma^{-\gamma}\,\bigg(\rho^{-n}\,\int\limits_{B_{0}}\,u^{\theta}\,dx\bigg)^{\frac{1}{\theta}}.
\end{equation*}
Choosing $\varepsilon_{0}$, $\varepsilon$ small enough, iterating this inequality and taking into account our choices we arrive at
\begin{equation*}
\max\limits_{B_{\rho/2}(x_{0})} u \leqslant \gamma\,
\bigg(\rho^{-n}\int\limits_{B_{0}}u^{\theta}\,dx\bigg)^{\frac{1}{\theta}} +\gamma\,\rho,
\end{equation*}
provided that \eqref{eq4.9} is valid. This proves inequality \eqref{eq1.11}.

Collecting estimates \eqref{eq1.9}, \eqref{eq1.11} with $\theta=\frac{1}{2}(p-1),$ we arrive at
\begin{equation*}
\max\limits_{B_{\rho/2}(x_{0})} u \leqslant \gamma\, \big(\min\limits_{B_{\rho/2}(x_{0})} u +
\rho\big),
\end{equation*}
which completes the proof of Theorem \ref{th1.3}.




\vskip3.5mm
{\bf Acknowledgements.} The research of the authors was supported by Project 0120U100178 from the National Academy of Sciences of Ukraine.

\newpage

CONTACT INFORMATION

\medskip

\medskip
\textbf{Igor I.~Skrypnik}\\Institute of Applied Mathematics and Mechanics,
National Academy of Sciences of Ukraine, \\ \indent Batiouk Str. 19, 84116 Sloviansjk, Ukraine\\
Vasyl' Stus Donetsk National University,
\\ \indent 600-richcha Str. 21, 21021 Vinnytsia, Ukraine\\ihor.skrypnik@gmail.com

\medskip
\textbf{Yevgeniia A. Yevgenieva}
\\ Max Planck Institute for Dynamics of Complex Technical Systems, \\ \indent Sandtorstrasse 1, 39106 Magdeburg, Germany
\\Institute of Applied Mathematics and Mechanics,
National Academy of Sciences of Ukraine, \\ \indent Batiouk Str. 19, 84116 Sloviansjk, Ukraine\\yevgeniia.yevgenieva@gmail.com

\end{document}